%% file: ediss-lower.tex
\documentclass{amsart}
\input{preamble.tex}

\begin{document}
\author[Cooperman]{William Cooperman}
\address{%
  Department of Mathematics, ETH Z\"urich.
}
\email{bill@cprmn.org}
\author[Iyer]{Gautam Iyer}
\address{%
  Department of Mathematical Sciences, Carnegie Mellon University, Pittsburgh, PA 15213.
}
\email{gautam@math.cmu.edu}
\author[Rowan]{Keefer Rowan}
\address{%
  Courant Institute of Mathematical Sciences, New York University, New York, NY, 10012
}
\email{keefer.rowan@cims.nyu.edu}
\author[Son]{Seungjae Son}
\address{%
  Department of Mathematical Sciences, Carnegie Mellon University, Pittsburgh, PA 15213.
}
\email{seungjas@andrew.cmu.edu}
\thanks{This work has been partially supported by the National Science Foundation under grants
  DMS-2342349, 
  DMS-2406853, 
  DMS-2303355, 
  and the Center for Nonlinear Analysis.}
\subjclass{%
  Primary:
    60J25, 
  Secondary:
    35Q49, 
    76R05. 
  }
\keywords{enhanced dissipation, mixing}
\title{Exponentially mixing flows with slow enhanced dissipation}
\begin{abstract}
  Consider a passive scalar which is advected by an incompressible flow~$u$ and has small molecular diffusivity~$\kappa$.
  Previous results show that if~$u$ is exponentially mixing and $C^1$, then the \emph{dissipation time} is~$O(\abs{\log \kappa}^2)$.
  We produce a family of incompressible flows which are~$C^0$ and exponentially mixing, uniformly in~$\kappa$; however have a dissipation time of order~$1/\kappa$ (i.e.\ exhibits no enhanced dissipation).
  We also estimate the dissipation time of mixing flows, and obtain improved bounds in terms of the mixing rate with explicit constants, and allow for a time inhomogeneous mixing rate which is typical for random constructions of mixing flows.
\end{abstract}
\maketitle

\section{Introduction}

\subsection{Main result}
The aim of this paper is to produce an example of an incompressible flow which mixes exponentially but does not enhance dissipation.
We begin by stating our results.
Following this, we will survey the literature and place our work in the context of existing results.

Let $\kappa > 0$, $u = u^\kappa$ be a (possibly time and~$\kappa$ dependent) divergence-free vector field on the~$d$-dimensional torus, and let~$\theta^\kappa$ solve the advection diffusion equation
\begin{equation}\label{e:ad}
  \partial_t \theta^\kappa + u^\kappa \cdot \grad \theta^\kappa - \kappa \lap \theta^\kappa = 0
  \,.
\end{equation}
Our interest is in understanding the asymptotic behavior of the \emph{dissipation time} $\tdis = \tdis(u, \kappa)$ as~$\kappa \to 0$.
Recall the \emph{dissipation time}~\cite{FannjiangWoowski03,FannjiangNonnenmacherEA04} is a measure of the rate at which solutions to~\eqref{e:ad} converge to their equilibrium state and is defined to be the smallest time after which mean-zero solutions to~\eqref{e:ad} are guaranteed to dissipate half of their initial~$L^2$ energy.
Since the flows we consider are time-inhomogeneous,
\begin{equation}
  \tdis(u, \kappa) \defeq \sup_{s \geq 0} \tdis^s(u, \kappa)
  \,,
\end{equation}
where for every~$s \geq 0$, $\tdis^s(u, \kappa)$ is defined by
\begin{equation}
  \tdis^s(u^\kappa, \kappa)^s \defeq \inf \set[\Big]{
    t \geq 0 \st
    \text{for every } \theta^\kappa_{s} \in \dot L^2,
    \text{ we have }
    \norm{\theta^\kappa_{s + t}}_{L^2} \leq \frac{\norm{\theta^\kappa_s}_{L^2}}{2} 
  }
  \,.
\end{equation}
Here the space~$\dot L^2 \subseteq L^2(\T^d)$ is the space of mean-zero square integrable functions on the torus~$\T^d$.
Moreover, the function~$\theta^\kappa$ in the in the infimum above is a solution to~\eqref{e:ad} with initial data~$\theta^\kappa_s$, specified at time~$s$.

Using the Poincar\'e inequality,~$\dv u^\kappa = 0$, and that $\theta \in \dot L^2$ quickly shows that the dissipation time is always bounded above by
\begin{equation}\label{e:tdisPoincare}
  \tdis(u^\kappa, \kappa) \leq \frac{\log 2}{\lambda_1 \kappa}
  \,,
\end{equation}
where~$\lambda_1$ is the smallest nonzero eigenvalue of the negative Laplacian.
In many physical situations, however, it is observed that the dissipation time is much smaller than the upper bound in~\eqref{e:tdisPoincare}; see~\cite{CotiZelatiCrippaEA24} for a review of the relevant literature.
This phenomenon is called \emph{enhanced dissipation}, and occurs, for example, when~$u^\kappa$ is mixing.
In fact, if~$u^\kappa$ is~$C^1$ and \emph{exponentially mixing} uniformly in~$\kappa$ then~$\tdis \leq C \abs{\log \kappa}^2$ (see Corollary~\ref{cor:exponential mixing dissipation enhancement}, below).
The aim of this paper is to produce an example of flows that are~$C^0$ and exponentially mixing uniformly in~$\kappa$ but \emph{do not} exhibit enhanced dissipation -- that is they saturate the bound~\eqref{e:tdisPoincare}.
(The flows we construct are in fact smooth, but higher norms are not bounded uniformly in~$\kappa$.)
Explicitly, we prove the following theorem.

\begin{theorem}\label{t:ediss-slowR}
  For every sufficiently small~$\kappa > 0$, there exists an incompressible velocity field~$u^\kappa$ (which will be constructed explicitly in Section~\ref{s:emuk}, below) such that the following hold:
  \begin{enumerate}[(1)]\reqnomode
    \item \label{i:uniform-C1R}
      \emph{(Uniform~$L^\infty$ boundedness)}
      There exists constants~$C, F$ that are independent of~$\kappa$ such that for all $\kappa>0$
      \begin{equation}\label{e:uniformC1R}
	\norm{u^\kappa}_{L^\infty([0, \infty) \times \T^d)} \leq C
	\quad\text{and}\quad
	 \norm{u^\kappa}_{L^\infty([0, \infty); C^1(\T^d) )}
	  < \frac{F}{\kappa}
	\,,
      \end{equation}
    \item \label{i:exp-mixingR}
      \emph{(Exponential mixing)}
      There exists constants~$D, \gamma_1 > 0$ (independent of~$\kappa$) such that for all sufficiently small~$\kappa > 0$, every solution to the transport equation
      \begin{equation}\label{e:transport}
	\partial_t \phi + u^\kappa \cdot \grad \phi = 0
	\,,
      \end{equation}
      with initial data~$\phi_0 \in \dot H^1$ satisfies the mixing bound for all $n,m \in 2\N$
      \begin{equation}\label{e:ukappa-mixrateR}
	\norm{\phi_{m+n}}_{H^{-1}}
	  \leq D (m^2+1) \exp\paren[\big]{-\gamma_1 n} \norm{\phi_m}_{H^1}
	  \,.
      \end{equation}

    \item \label{i:slow-dissipationR}
      \emph{(No enhanced dissipation)}
      For all sufficiently small~$\kappa > 0$, and for all~$s \geq 0$, the dissipation time satisfies the lower bound
      \begin{equation}\label{e:tdis-lowerR}
	\tdis^s (u^\kappa, \kappa) \geq \frac{C}{\kappa}
	\,.
      \end{equation}
  \end{enumerate}
\end{theorem}
\begin{remark}
    We note that our main mixing estimate~\eqref{e:ukappa-mixrateR} is somewhat non-standard in two distinct ways. The first is that the estimate is only given on times $s,t \in 2 \N$. For a uniformly Lipschitz velocity field, a statement of mixing on times $s,t \in 2\N$ is equivalent to a statement for times $s,t \in [0,\infty)$ using the standard estimates for the fluctuation of the $H^1$ and $H^{-1}$ norms of the solution to~\eqref{e:transport}, paying only an additional prefactor constant.

    In our setting we do not have uniform-in-$\kappa$ bounds on the Lipschitz norm of $u$, and so~\eqref{e:ukappa-mixrateR} does not immediately imply a similar estimate for all $s,t \in [0,\infty)$.
    However, using the construction of~$u^\kappa$ (specifically, using the shear structure) and the Lipschitz norm bound in~\eqref{e:uniformC1R} and the integer time mixing bound~\eqref{e:ukappa-mixrateR}, we can readily deduce the mixing bound
    \begin{equation}\label{e:ukappa-mixrate-st}
    \norm{\phi_{t+s}}_{H^{-1}}
	  \leq \frac{D (s^2+1)}{\kappa^2} \exp\paren[\big]{-\gamma_1 t} \norm{\phi_s}_{H^1}
	  \quad\text{ for all } s,t\in[0,\infty),
	  \,.
    \end{equation}
    We note that~\eqref{e:ukappa-mixrate-st} has a~$\kappa$-dependent pre-factor, it is unimportant in the study of enhanced dissipation as it only enters logarithmically (see Corollary~\ref{cor:exponential mixing dissipation enhancement}, below).

    The second way in which our main mixing estimate~\eqref{e:ukappa-mixrateR} differs from the usual statement of mixing is the lack of initial time homogeneity: the prefactor constant depends on the initial time $m$.
    This is however a feature of all currently known smooth mixing flows and a consequence of their random construction, and is discussed further in Section~\ref{s:mixing homogeneity}, below.
\end{remark}

The main idea behind the construction is to alternate randomly shifted horizontal and vertical shears, inspired by the construction of Pierrehumbert~\cite{Pierrehumbert94}.
In order to ensure that these flows \emph{do not} enhance dissipation (in the sense of~\eqref{e:tdis-lowerR}), we choose shear profiles with amplitude~$1$ that oscillate on scales of order~$\kappa$.
Based on energy estimates we show that any such flow can not enhance dissipation in the sense of~\eqref{e:tdis-lowerR}.

Proving the~$\kappa$-independent exponential mixing bound~\eqref{e:ukappa-mixrateR} is more delicate.
Existing methods verify Harris conditions for the two point chain~\cite{BedrossianBlumenthalEA21,BlumenthalCotiZelatiEA23}, and then use a Borel--Cantelli argument as in~\cite{DolgopyatKaloshinEA04} to prove almost sure exponential mixing.
The Harris conditions require a Lyapunov function, which is typically produced using a version of Furstenberg's criterion~\cite{Furstenberg63}.
This \emph{can not} be used in our situation because this approach will only yield a~$\kappa$-dependent Lyapunov function, and a~$\kappa$-dependent Lyapunov exponent, resulting in an exponential mixing estimate like~\eqref{e:ukappa-mixrateR} where the constant~$\gamma_1$ depends on~$\kappa$ in an uncontrolled manner.
This is not suitable as it will not address sharpness of dissipation time bounds (described in Section~\ref{s:sharpness}, below).
Moreover, if~$\gamma_1$ vanishes fast with~$\kappa$ the dissipation time lower bound~\eqref{e:tdis-lowerR} is not surprising, nullifying main point of Theorem~\ref{t:ediss-slowR}.

We will instead prove Theorem~\ref{t:ediss-slowR} by constructing a~$\kappa$-independent Lyapunov function explicitly (see Section~\ref{s:emuk}, below).
Following this, existing methods~\cite{DolgopyatKaloshinEA04} can be used to prove almost sure exponential mixing as in~\eqref{e:ukappa-mixrateR} with a~$\kappa$-independent (deterministic) rate constant~$\gamma_1$, and a~$\kappa$-dependent (random) prefactor~$D_\kappa$.
Using the fact that the distribution of~$u$ is stationary, and that moments of~$D_\kappa$ are bounded uniformly in~$\kappa$ we will show the exponential mixing bound~\eqref{e:ukappa-mixrateR}.

\subsection{Time inhomogeneous mixing}\label{s:mixing homogeneity}

Let us further discuss the lack of time homogeneity in our main mixing result~\eqref{e:ukappa-mixrateR}. We first introduce our definition of a mixing rate. Let~$h \colon [0, \infty)^2 \to (0, \infty)$ be a function such that for every~$s$, the function~$t \mapsto h(s, t)$ is decreasing and vanishes at infinity.
We say the velocity field~$u$ is \textit{mixing with rate~$h$} if for every~$s, t \geq 0$, and every $\phi_s \in \dot H^1$, the solution to~\eqref{e:transport} with~$u^\kappa = u$ and initial data~$\phi_s$ at time~$s$ satisfies the mixing bound
\begin{equation}\label{e:mixrate}
  \norm{\phi_{s + t}}_{H^{-1}} \leq h(s, t) \norm{\phi_s}_{H^1} \,.
\end{equation}
We say that the mixing is \emph{time homogeneous} if $h(s,t)$ is independent of $s$.
We say~$u$ is \emph{exponentially mixing} if the mixing rate is of the form~$h(s, t) = D(s) e^{-\gamma t}$ for some function~$D$ that only depends on the initial time~$s$.

We remark that in context of dynamical systems, the velocity field~$u$ is time-independent, in which case the mixing is time homogeneous ``for free''.
In the time dependent setting, the mixing rate generically has a nontrivial dependence on the initial time $s$. In fact, surveying some of the myriad mixing examples so far constructed---\cite{
  YaoZlatos17,
  AlbertiCrippaEA19,
  ElgindiZlatos19,
  BedrossianBlumenthalEA22,
  MyersHillSturmanEA22,
  BlumenthalCotiZelatiEA23,
  CotiZelatiNavarroFernandez24,
  ElgindiLissEA25
}---only~\cite{MyersHillSturmanEA22,ElgindiLissEA25}, built using deterministic tools and explicitly time periodic, have a mixing rate that is both time homogeneous, and exponentially decaying.
Interestingly, these examples are only Lipschitz regular, and consequently the following problem remains open.
\begin{problem}
  Construct $u \in L^\infty([0,\infty), C^{1,\alpha}(\T^d))$ for some $\alpha>0$, which has a time-homogeneous exponential mixing rate~$h(s, t) = D e^{-\gamma t}$ for some constants~$D < \infty$ and~$\gamma > 0$ that are independent of~$s$ and~$t$.
\end{problem}

We note that the random constructions of exponential mixers~\cite{BedrossianBlumenthalEA22,BlumenthalCotiZelatiEA23,CotiZelatiNavarroFernandez24} only state exponential mixing for $s=0$: that is they only state a mixing result for initial data specified at time~$s = 0$.
For initial data specified at an arbitrary time~$s > 0$, it is natural to expect that the mixing rate~$h(s, t)$ has some dependence on~$s$.
Indeed, on any time interval, there is a small probability that the mixing is highly ineffective, and so on large time intervals we must pick up some periods where the mixing is less effective.
\smallskip

However, as their random constructions are time stationary in law, one can use a Borel--Cantelli type argument to prove exponential mixing for all initial times $s\geq 0$, at the cost of an $s$-dependent prefactor.
We use this argument in Section~\ref{s:proof of theorem}, below, to derive our mixing estimate~\eqref{e:ukappa-mixrateR}.
The same argument gives a generalization of main result of~\cite{CoopermanIyerEA25}, which we now state.
\begin{proposition}\label{p:mixInhomog}
  Suppose the flow of~$u$ generates a random dynamical system that satisfies Assumptions~2.1--2.4 of~\cite{CoopermanIyerEA25}.
  Additionally, suppose the law of~$u$ is stationary in time.
  There exists a random $D \geq 1$, and deterministic~$\gamma > 0$ such that we have the almost sure mixing bound~\eqref{e:mixrate} with
  \begin{equation}\label{e:expdecay}
    h(s, t) = D (1 + s^2) e^{-\gamma t}
    \,.
  \end{equation}
  Moreover, for every~$q < \infty$, $\alpha > 0$, and all sufficiently small~$\kappa > 0$, there exists a random~$D_{q, \kappa} \geq 1$ and finite constants~$\bar D_q$ (independent of~$\kappa, \alpha$), $C_\alpha$ (independent of~$q, \kappa$) such that for every~$s, t \geq 0$, $\theta_s \in \dot L^1$ the solution to~\eqref{e:ad} with initial data~$\theta_s$ at time~$s$ satisfies
  \begin{equation}\label{e:edissInhomog}
  \norm{\theta^\kappa_{s + t}}_{L^\infty} \leq \frac{C_\alpha D_{q, \kappa} (1 + s^2) e^{-\gamma t}}{\kappa^{\frac{d}{2} + \alpha}}  \norm{\theta^\kappa_s}_{L^1}
  \,,
\end{equation}
\end{proposition}

Clearly Proposition~\ref{p:mixInhomog} implies the dissipation time bound
\begin{equation}\label{e:tdiss}
  \tdis^s \leq \ln \paren[\Big]{
    \frac{D_{q, \kappa} ( 1 + s^2) }{\kappa}
  }
  \quad\text{and}\quad
  \E D_{q, \kappa}^q \leq \bar D_q
  \,,
\end{equation}
but does not imply a bound on~$\tdis$.

\subsection{Enhanced dissipation}

Given a velocity field is mixing with time-homogeneous rate~$h(s,t) = h(t)$, existing results~\cite{Feng19,FengIyer19,CotiZelatiDelgadinoEA20} provide an upper bound for the dissipation time. However no such bound has been given in the setting of time-inhomogeneous mixing, which appears naturally in the random mixing constructions. A refinement of these results, with explicit constants, a better dependence on the mixing rate, and accounting for non-homogeneity in the initial time is given below.

The central refinement is given by the following proposition, which gives that on a shorter time scale ($\abs{\log \kappa}$ for exponentially mixing flows) we get \emph{some} dissipation ($\abs{\log \kappa}^{-1/2}$ much for exponentially mixing flows). It is then by iterating this result that we bound the dissipation time and deal with the time-inhomogeneity.

\begin{proposition}\label{prop:dissipation at mixing time}
    For all $\kappa>0, \theta_0 \in \dot L_2$, if $\theta^\kappa_t$ is the solution~\eqref{e:ad} with incompressible velocity field $u$ with mixing rate $h(s,t)$, then  
    \[\|\theta^\kappa_{\tau_\kappa}\|_{L^2} \leq \sqrt{1 - A_\kappa} \|\theta_0\|_{L^2},\]
    where
    \begin{align*}
    H(T) &\defeq \sup_{0 \leq s \leq T/3 \leq t \leq T} h(s,t)
    \\\tau_\kappa &\defeq \inf \big\{t \in [2,\infty) : t^{-2} H(t) \leq 2^{8}(\|\nabla u\|_{L^\infty_{t,x}}+1)\kappa\big\}, \\A_\kappa &\defeq \frac{1}{2^{15}(\|\nabla u\|_{L^\infty_{t,x}} + 1) \tau_\kappa}.
    \end{align*}
\end{proposition}

A careful iteration of the above gives the following after some computation. 

\begin{corollary}
\label{cor:exponential mixing dissipation enhancement}
    Suppose $u \in L^\infty([0,\infty); W^{1,\infty}(\T^d))$ is an incompressible velocity field with mixing rate $h(s,t) \leq D (1+s^p) e^{-\gamma t}$ for some $p \in [0,\infty)$. Then for all $\kappa >0$,
    \begin{align}
      \tdis^0(u,\kappa) &\leq 2^{24}(\|\nabla u\|_{L^\infty_{t,x}} + 1) \cdot
      \\
    \label{e:tdis-expmix}
    &\qquad \cdot \bigg(1 +  2^{24} \Big(\frac{p}{\gamma}\Big)^4(\|\nabla u\|_{L^\infty_{t,x}} + 1) + \frac{(\log D)^2 + \abs{\log \kappa}^2}{\gamma^2}  \bigg).
    \end{align}
\end{corollary}

The proof of these results will be given in Section~\ref{s:tdiss-bound}.

\subsection{Sharpness of dissipation time bounds.}\label{s:sharpness}

As a consequence of Theorem~\ref{t:ediss-slowR}, we can address (though not completely resolve) a gap between the heuristically expected bounds on the dissipation time and the rigorously provable bounds given by Corollary~\ref{cor:exponential mixing dissipation enhancement}. Of the exponential mixing examples, the dissipation times for~\cite{
  BedrossianBlumenthalEA21,
  CoopermanIyerEA25,
  ElgindiLissEA25
} are known.
Interestingly, in each case the dissipation time is~$\tdis^0 \approx \abs{\log \kappa}$, which is a \emph{stronger bound} than the bound given by Corollary~\ref{cor:exponential mixing dissipation enhancement}: $\tdis^0 \leq C \abs{\log \kappa}^2$.
Moreover, a compelling heuristic argument (which we explain in Section~\ref{s:review}, below) suggests that flows which are sufficiently regular (uniformly in~$\kappa$) and are mixing with rate $h(s,t) = D e^{-\gamma t}$ should in fact generically satisfy the stronger dissipation time bound
\begin{equation}\label{e:tdis-upper-heuristic}
  \tdis(u^\kappa, \kappa)
    \leq
      \frac{1}{\gamma}
	  \abs[\bigg]{\log\paren[\bigg]{\frac{C_\delta(u) D}{\kappa^{\frac{d}{2} + \delta}}}}
  \,,
\end{equation}
for any~$\delta > 0$.

Theorem~\ref{t:ediss-slowR} can be used to show that the gap between the provable bound of Corollary~\ref{cor:exponential mixing dissipation enhancement} and the expected bound~\eqref{e:tdis-upper-heuristic} may not be closable.
To explain this, let~$\set{u^\kappa}$ be the family of flows from Theorem~\ref{t:ediss-slowR}.
Rescale time by a factor of~$\kappa$, and set
\begin{equation}\label{eq:vkdef}
  v^\kappa_t(x) = \epsilon u^{\epsilon}_{\epsilon t}( x )
  \quad\text{where }
  \epsilon = \sqrt{\kappa}
  \,.
\end{equation}
Clearly, the velocity field~$v^\kappa$ satisfies the~$\kappa$-uniform $C^1$ bound
\begin{equation}\label{eq:vkbound}
  \norm{v^\kappa}_{L^\infty([0, \infty) \times \T^d)} \leq C \sqrt{\kappa}
  \quad\text{and}\quad
  \sup_{\kappa > 0} \norm{v^\kappa}_{L^\infty([0, \infty); C^1(\T^d) )}
    < F
  \,,
\end{equation}
for some constants~$C, F$ independent of~$\kappa$.
The velocity field~$v^\kappa$ is still exponentially mixing; however, due to the time rescaling the mixing rate is now~$\kappa$ dependent.
That is, every solution to the transport equation~\eqref{e:transport} (with drift~$v^\kappa$, instead of~$u^\kappa$) satisfies the mixing bound
\begin{equation}\label{e:ukappa-mixrate}
  \norm{\phi_t}_{H^{-1}(\T^d)}
    \leq D \exp\paren[\big]{-\gamma_\kappa t} \norm{\phi_0}_{H^1}
    \quad\text{where } \gamma_\kappa \defeq \gamma_1 \sqrt{\kappa}
    \,.
\end{equation}

Finally, we note that the time rescaling gives
\begin{equation}\label{e:tdisV}
  \tdis( v^\kappa, \kappa )
    = \frac{1}{\epsilon} \tdis( u^\epsilon, \epsilon )
    \overset{\eqref{e:tdis-lowerR}}{\geq}
      \frac{C}{\epsilon^2}
    = \frac{C}{\kappa}
    \,.
\end{equation}
Using~\eqref{e:ukappa-mixrate} we see that the right hand side of~\eqref{e:tdis-upper-heuristic} is only of order~$1/\sqrt{\kappa}$ (up to a logarithmic factor), which is an order of magnitude smaller than the right hand side of~\eqref{e:tdisV}.
Thus,
this gives a family of exponentially mixing flows~$v^\kappa$ which are~$C^1$ bounded (uniformly in~$\kappa$) for which the stronger bound~\eqref{e:tdis-upper-heuristic} can not hold.

We note that the above isn't a proof that Corollary~\ref{cor:exponential mixing dissipation enhancement} is sharp as our example is a~$\kappa$-dependent family.
Even though these flows are~$C^1$ bounded uniformly in~$\kappa$, our construction does not allow us to find one, $\kappa$-independent, velocity field for which we have a matching lower bound to Corollary~\ref{cor:exponential mixing dissipation enhancement}.

\subsection{Literature review}\label{s:review}
We now survey the literature and place our results in the context of existing results.
\emph{Enhanced dissipation} is a phenomenon that can be observed in every day life:
pour some cream in your coffee.
If left alone, it will take hours to mix.
Stir it a little and it mixes right away.
This effect arises due to the interaction between the advection (stirring) and diffusion, and it plays an important role in many applications concerning hydrodynamic stability and turbulence and occurs on scales ranging from microfluidics to the meteorological and cosmological~\cite{
  LinThiffeaultEA11,
  Thiffeault12,
  Aref84,
  StoneStroockEA04
}.

To describe this mathematically, let $u$ be the velocity field of the ambient incompressible fluid and~$\rho^\kappa$ denote the concentration of a passively advected solute with molecular diffusivity $\kappa > 0$.
The evolution of~$\rho^\kappa$ is governed by the advection diffusion equation
\begin{equation}\label{e:ad1}
  \partial_t \rho^\kappa + u \cdot \grad \rho^\kappa - \kappa \lap \rho^\kappa = 0
  \,.
\end{equation}
For simplicity, in this paper we only consider~\eqref{e:ad1} with periodic boundary conditions on the $d$-dimensional torus $\T^d$.

If the ambient fluid is incompressible, the velocity field~$u$ is divergence-free (i.e.\ $\dv u = 0$).
In this case, an elementary energy estimate and the Poincar\'e inequality shows that the $L^2$ distance of the concentration from the equilibrium distribution decreases at most exponentially with a rate proportional to~$\kappa$, and yields the dissipation time bound~\eqref{e:tdisPoincare}.

This, however, is a very crude upper bound and completely neglects the effect of the advection.
In practice, the advection term typically causes filamentation and moves mass towards small scales.
Since the diffusion term $\kappa \lap \rho^\kappa$ damps small scales faster, the combined effect of advection and diffusion should lead to faster convergence of~$\rho^\kappa_t$ to its equilibrium state as~$t \to \infty$.
This phenomenon is known as \emph{enhanced dissipation} and has been observed in many situations.

One indication of enhanced dissipation is when~$\tdis(u, \kappa) = o(1/\kappa)$, and has been studied by several authors.
Seminal work of Constantin et\ al.~\cite{ConstantinKiselevEA08} (see also~\cite{Zlatos10,KiselevShterenbergEA08}) shows that if $u$ is time independent, then~$\kappa \tdis(\kappa) \to 0$ if and only if $u \cdot \grad$ has no eigenfunctions in $H^1$.
For shear flows, the classical work of Kelvin~\cite{Kelvin87} shows~$\tdis(\kappa) = \kappa^{-\alpha}$ for some $\alpha < 1$.
There are now several results studying enhanced dissipation in more generality and for nonlinear equations (see~\cite{
  FannjiangNonnenmacherEA04,
  BedrossianCotiZelati17,
  Wei19,
  CotiZelatiDrivas21,
  AlbrittonBeekieEA22,
  FengShiEA22,
  FengMazzucatoEA23,
  CotiZelatiGallay23,
  CobleHe23,
  FengHuEA23,
  Seis23,
  IyerZhou23,
  TaoZworski25
  }).
\GI[2025-06-18]{TODO: Cite a few more recent ones.}

Of particular interest is the connection between \emph{mixing} and enhanced dissipation, which has been studied and quantified in~\cite{
  Feng19,
  FengIyer19,
  FengFengEA20,
  CotiZelatiDelgadinoEA20,
  CotiZelatiCrippaEA24
}.
In particular, it is known that for exponentially mixing flows, we have~$\tdis(\kappa) \leq C \abs{\log \kappa}^2$, as stated in Corollary~\ref{cor:exponential mixing dissipation enhancement}.
The main result of this paper (Theorem~\ref{t:ediss-slowR}) produces a family of~$\kappa$-dependent incompressible vector fields~$u^\kappa$ which are~$C^0$ bounded and exponentially mixing uniformly in~$\kappa$, however are not dissipation enhancing (i.e.\ saturate~\eqref{e:tdisPoincare}).

Another point of interest is optimality of the dissipation time bound for mixing flows (Corollary~\ref{cor:exponential mixing dissipation enhancement}).
In particular, if~$u$ is~$C^1$ and exponentially mixing, it is not known if the~$O(\abs{\log \kappa}^2)$ bound given by Corollary~\ref{cor:exponential mixing dissipation enhancement} is sharp.
Indeed, to the best of our knowledge, for all known examples of exponentially mixing flows where the dissipation time can be computed, it is bounded by~$O(\abs{\log \kappa})$.
These examples include
  toral automorphisms~\cite{FannjiangNonnenmacherEA04,FengIyer19},
  stochastically forced Navier--Stokes equations~\cite{BedrossianBlumenthalEA21},
  expanding Bernoulli maps~\cite{IyerLuEA24},
  alternating shear flows~\cite{ElgindiLissEA25},
  contact Anosov flows~\cite{TaoZworski25},
  Pierrehumbert flows~\cite{CoopermanIyerEA25}.
The last reference~\cite{CoopermanIyerEA25} in fact shows that if~$u$ is generated by a random dynamical system that satisfies certain Harris conditions, then the dissipation time is bounded by~$O(\abs{\log \kappa})$.
These Harris conditions also imply almost sure exponential mixing, and so~\cite{CoopermanIyerEA25} provides a large class of flows (including Pierrehumbert flows) which are exponentially mixing and enjoy an~$O(\abs{\log \kappa})$ dissipation time bound.

A heuristic explanation for the~$O(\abs{\log \kappa})$ dissipation time bound is as follows.
The SDE associated with~\eqref{e:ad} is
\begin{equation}\label{e:SDEX}
  d X^\kappa_t = u_t(X^\kappa_t) \, dt + \sqrt{2 \kappa} \, dW_t
  \,,
\end{equation}
where~$W$ is a standard~$d$-dimensional Brownian motion.
Note that when~$\kappa = 0$, the process $X^0$ is simply the deterministic flow of the vector field~$u$.
Thus, if~$u$ is exponentially mixing with rate~$h_t = D e^{-\gamma t}$, then the mix norm decay~\eqref{e:mixrate} is equivalent to
\begin{equation}\label{e:X0mixing}
  \norm{\phi \circ X^0_t}_{H^{-1}}
    \leq D e^{-\gamma t} \norm{\phi}_{H^1}
  \quad
  \text{for all test functions }
  \phi \in \dot H^1(\T^d)
  \,.
\end{equation}

Since the noise is stationary, ergodic and homogeneous, it is reasonable to expect that the mixing estimate~\eqref{e:X0mixing} persists for small~$\kappa > 0$.
That is, for all~$\kappa > 0$ we still expect to have the \emph{almost sure} mixing bound
\begin{equation}\label{e:XkappaMixing}
  \norm{\phi \circ X^\kappa_t}_{H^{-1}}
    \leq D_1 e^{-\gamma_1 t} \norm{\phi}_{H^1}
  \quad
  \text{for all test functions }
  \phi \in \dot H^1(\T^d)
  \,.
\end{equation}
and some~$\kappa$-independent constants~$D_1$, $\gamma_1$.
Proving this rigorously is open (see Question~1.9 in~\cite{CoopermanIyerEA25}).
However, if~$u$ satisfies either the stochastic Navier--Stokes equations, or certain Harris conditions, it is possible to show~\cite{BedrossianBlumenthalEA21,CoopermanIyerEA25}  both exponential mixing as in~\eqref{e:X0mixing} and uniform in~$\kappa$ exponential mixing as in~\eqref{e:XkappaMixing} for all small~$\kappa$.

Now if the uniform mixing estimate~\eqref{e:XkappaMixing} holds, then the Kolmogorov backward equation, parabolic regularity, and the fact that~$\norm{\E \phi \circ X^\kappa_t}_{H^{-1}} \leq \E \norm{ \phi \circ X^\kappa_t }_{H^{-1}}$ can be used to obtain a fast decay estimate for solutions to~\eqref{e:ad1}.
Indeed following the proof of Theorem~1.1 in~\cite{CoopermanIyerEA25} shows that for every~$\alpha > 0$, there exists a constant~$C_\alpha = C_\alpha(u) > 0$ such that every solution of~\eqref{e:ad1} satisfies
\begin{equation}\label{e:rho-L1-Linf}
  \norm{\rho_t - \bar \rho}_{L^\infty} \leq \frac{C_\alpha D_1}{\kappa^{\frac{d}{2} + \alpha}}
    e^{- \gamma_1 t}
    \norm{\rho_0 - \bar \rho }_{L^1} \,,
  \quad\text{where}\quad
  \bar \rho = \int_{\T^d} \rho_0 \, dx
  \,.
\end{equation}
This immediately implies the dissipation time bound~\eqref{e:tdis-upper-heuristic}.
\smallskip

The long time decay of solutions to~\eqref{e:ad1} provides another interesting avenue for study.
In turbulence theory~\cite{Batchelor59} identified a length scale at which the effects of advection and diffusion balance.
At this scale~$\kappa \norm{\grad \rho^\kappa}_{L^2}$ becomes \emph{independent} of~$\kappa$ and~$\rho^\kappa_t \to \bar \rho$ at a rate that is independent of the diffusivity (see~\cite{Thiffeault12,MilesDoering18}).
This is reflected in~\eqref{e:rho-L1-Linf}, which provides a \emph{$\kappa$-independent} decay rate for large time.
Notice that this~$\kappa$-independent rate can not be obtained with knowledge of the dissipation time alone.
To the best of our knowledge, there are no continuous time deterministic examples that exhibit this~$\kappa$-independent long time decay rate.


On a different note, if $u$ is regular, and uniformly bounded in time, then we must have $\tdis(u,\kappa) \geq O(\abs{\log \kappa})$ (see for instance~\cite{MilesDoering18,Poon96,BedrossianBlumenthalEA21,Seis22}).
When $u$ is irregular, this is no longer the case and one can even have $\tdis(\kappa) = O(1)$ as~$\kappa \to 0$.
This is known as~\emph{anomalous dissipation} and was predicted in certain turbulent regimes Obukhov and Corrsin~\cite{Obukhov49,Corrsin51}.
Examples of this were constructed recently in~\cite{
  DrivasElgindiEA22,
  ColomboCrippaEA22,
  ArmstrongVicol23,
  BrueColomboEA24,
  BurczakSzekelyhidiEA24,
  ElgindiLiss24,
  JohanssonSorella24,
  HessChildsRowan25
}.

Finally we mention that the dissipation time~$\tdis$ is dual to the~$L^2$ \emph{mixing time}.
The mixing time of various Markov processes has been an active area of study for decades~\cite{
  ChenLovaszEA99,
  Diaconis09,
  LevinPeres17
}.
One relevant discrete result is~\cite{ChatterjeeDiaconis20}, where the authors study the mixing time of a random walk with deterministic jumps specified by a bijection.
Two other relevant are~\cite{
  DamakFrankeEA20,
  ChristieFengEA23
}
where the authors use a measure preserving drift to accelerate convergence of an Ornstein--Ullenbeck process.

\subsection*{Plan of this paper}

In Section~\ref{s:construction} we describe the construction of the velocity field~$u^\kappa$ and prove Theorem~\ref{t:ediss-slowR} and Proposition~\ref{p:mixInhomog} modulo two results.
The first (Lemma~\ref{l:kindd}) asserts exponential mixing of~$u^\kappa$ with rate~$h(t) = D_\kappa e^{-\gamma_1 t}$, where~$\gamma_1$ is independent of~$\kappa$, and~$D_\kappa$ is~$\kappa$-dependent and random, but has a~$\kappa$-independent moment bound.
The second (Proposition~\ref{p:tdis-lower}) asserts that no enhanced dissipation occurs for any velocity fields for which~$\norm{u^\kappa}_{W^{-1, \infty}} \leq O(\kappa)$.
Most of the paper is devoted to the proof of Lemma~\ref{l:kindd} and is presented in Section~\ref{s:emuk}.
The main ingredient of the proof is verifying Harris conditions with a controlled~$\kappa$-dependence.
In Section~\ref{s:slowm} we use energy methods to prove a lower bound on the dissipation time (Proposition~\ref{p:tdis-lower}).
Finally in Section~\ref{s:tdiss-bound} we prove Corollary~\ref{cor:exponential mixing dissipation enhancement}.
While the Corollary~\ref{cor:exponential mixing dissipation enhancement} itself is only a small improvement of existing results, the proof we present is much shorter and simpler. 

%

\section{Proof of Theorem \ref{t:ediss-slowR}}
\subsection{Construction of \texorpdfstring{$u^\kappa$}{ukappa}}\label{s:construction}

In this section, our aim is to describe the velocity field $u^\kappa$ which satisfies the properties of Theorem~\ref{t:ediss-slowR}.
We let the torus~$\T^2$ be the square~$[0, 2\pi]^2$ with opposite sides identified.
Define
\begin{equation}\label{e:Ndef}
  N_\kappa = \ceil*{\frac{1}{\kappa}}
  \,.
\end{equation}
For time~$0 \leq t \leq 1$, we will choose~$u^\kappa$ so that it consists of $N_\kappa$ horizontal shears supported on intervals of size $2\pi/N_\kappa$.
Each shear will be translated by a random phase shift with distribution~$\Unif([0, \pi/N_\kappa] )$.
For time $1\leq t \leq 2$, we will choose $u^\kappa$ by the same construction, but using vertical shears instead of horizontal ones, and independent random shifts.
For~$t \geq 2$ we will repeat this construction, alternating between horizontal and vertical shears, and choosing independent random phase shifts at each step.

To define the shears at each step, we consider an amplitude $A>0$, and a (piecewise)~$C^3$ profile $\varphi \colon [0, 2\pi] \to \R$.
For each time $m\in \N$, and each~$i \in \set{ 0, 1, \ldots, 2N_\kappa-1}$, we choose~$\alpha^m_i$ to be independent phase shifts with distribution
\begin{equation}
  \alpha_i^m \sim \Unif\paren[\Big]{\brak[\Big]{\frac{\pi i}{N_\kappa}, \frac{\pi (i+1)}{N_\kappa}}}
  \,.
\end{equation}
Define~$\varphi_\kappa$ to be the rescaled profile defined by
\begin{equation}\label{e:phiKappaDef}
  \varphi_\kappa(x)\defeq \varphi(N_\kappa x)
\,.
\end{equation}
We will chose the velocity field in Theorem~\ref{t:ediss-slowR} to be
\begin{equation}\label{eq:ukdef}
  u^\kappa_t(x) \defeq
    \begin{dcases}
      A\paren[\Big]{\sum_{i=0}^{2N_\kappa-1} \varphi_\kappa \paren[\Big]{x_2-\alpha_i^{2n}}}e_1 & t\in [2n, 2n+1)\,, \\
	A\paren[\Big]{\sum_{i=0}^{2N_\kappa-1} \varphi_\kappa \paren[\Big]{x_1-\alpha_i^{2n+1}}}e_2 & t\in [2n+1, 2n+2)\,.
    \end{dcases}
\end{equation}
Here~$e_1$, $e_2$ are the standard basis vectors in~$\R^2$.
Theorem~\ref{t:ediss-slowR} will be proved by showing that with positive probability, a choice of the random phase shifts~$\alpha^m_i$ will yield~\eqref{e:ukappa-mixrateR}.

For technical reasons we will assume $\varphi$ satisfies the following properties
\begin{enumerate}[({A}1)]
  \item\label{i:first-non-degenerate} 
    The set~$\set{\varphi' = 0}$ is finite, and~$\set{\varphi' = 0} \cap \set{\varphi'' = 0} = \emptyset$.
  \item \label{i:second-non-degenerate}
    The set~$\set{\varphi'' = 0}$ is finite, and~$\set{\varphi'' = 0} \cap \set{\varphi''' = 0} = \emptyset$.
\end{enumerate}

\subsection{Proof of Theorem \ref{t:ediss-slowR}}
\label{s:proof of theorem}

We will prove Theorem~\ref{t:ediss-slowR} in two steps.
First we will show that the velocity field defined in~\eqref{eq:ukdef} is almost surely exponentially mixing and satisfies~\eqref{e:ukappa-mixrateR} with~$m = 0$ and a random~$D_\kappa$ that depends on~$\kappa$, and a deterministic constant~$\gamma_1$ that is independent of~$\kappa$.
This is our first lemma.

\begin{lemma}
\label{l:kindd}
  Fix~$q < \infty$, and let~$u^\kappa$ be as in~\eqref{eq:ukdef}.
  For all sufficiently large~$A$, there exists a (deterministic) constant~$\gamma_1 > 0$ such that the following holds.
  For all sufficiently small~$\kappa > 0$, there exists a random~$D_\kappa > 0$ such that every solution to the transport equation~\eqref{e:transport} with initial data~$\phi_0 \in \dot H^1$ satisfies the mixing bound for all $n \in 2\N$
      \begin{equation}
      \label{e:ukappa-mixrate-kappa}
	\norm{\phi_n}_{H^{-1}}
	  \leq D_\kappa \exp\paren[\big]{-\gamma_1 n} \norm{\phi_0}_{H^1}
	  \,.
      \end{equation}
  Moreover, there exists a finite deterministic constant~$\bar D_q$, that is independent of~$\kappa$, such that
  \begin{equation}
  \label{e:Dbound}
    \E D_\kappa^q < \bar D_q
    \,.
  \end{equation} 
\end{lemma}

Next we will prove a general result that shows that if~$u^\kappa$ is~$O(\kappa)$ in~$W^{-1, \infty}$, then enhanced dissipation can not occur.
\begin{proposition}\label{p:tdis-lower}
    Suppose $u^\kappa$ is a family of velocity fields such that there exists a~$d \times d$ skew-symmetric matrix valued function~$H^\kappa$ such that
    \begin{equation}\label{e:w-1inftybound}
      u^\kappa_t = \dv H^\kappa_t
	\quad\text{and}\quad
        \norm{H^\kappa}_{L^\infty} \leq C_0 \kappa\,,
    \end{equation}
    for some constant~$C_0$.
    Then for every~$s \geq 0$ we have
    \begin{equation}\label{e:tdisLower}
        \tdis^s(u^\kappa, \kappa)\geq \frac{C_1}{\kappa}\,,
	\quad\text{where}\quad
      C_1 = \min\set[\Big]{
	  \frac{\log(4/3)}{2 \iffalse\lambda_1\fi},
	  \frac{1}{8 C_0^2 \iffalse\lambda_1\fi }
	}
	\,.
    \end{equation}
\end{proposition}

Momentarily postponing the proofs of Lemma~\ref{l:kindd} and Proposition~\ref{p:tdis-lower} we prove Theorem~\ref{t:ediss-slowR}. We introduce the following notation.
\begin{definition}\label{def:solution operator}
  For $0 \leq s \leq t$, $\kappa \geq 0$, and a flow $u \colon [0,\infty) \times \T^d \to \R^d$, let $\sol^{u,\kappa}_{s,t} : \dot L^2 \to \dot L^2$ be the solution operator to the advection diffusion equation~\eqref{e:ad}.
  That is, if~$\theta_s \in L^2$, then $\theta_t \defeq \sol^{u, \kappa}_{s, t} \theta_s$ is the unique solution to~\eqref{e:ad} with initial data~$\theta_s$ at time~$s$.
\end{definition}

\begin{proof}[Proof of Theorem~\ref{t:ediss-slowR}]
  \restartsteps\step
  We note that by the definition of $u^\kappa$~\eqref{eq:ukdef}, each $x_2 \in [\frac{\pi}{N_\kappa}i, \frac{\pi}{N_\kappa}(i+1)]$ can be hit by horizontal shears of phase shifts $\alpha_{i-2}, \alpha_{i-1}$, and $\alpha_i$ at time $[2n, 2n+1)$. This property directly implies
\begin{equation}
    \norm{u^\kappa}_{L^\infty([0,\infty \times \T^d])} \leq 3A\norm{\varphi}_\infty 
    \quad\text{and}\quad
    \norm{u^\kappa}_{L^\infty([0,\infty),C^1(\T^d) )} \leq 3AN_\kappa\norm{\varphi'}_\infty\,,
\end{equation}
  proving~\eqref{e:uniformC1R}.
  \smallskip

\step
  To prove exponential mixing with a~$\kappa$-independent rate, we use Lemma~\ref{l:kindd} with $q=1$, together with the time stationarity of the law of the velocity field $u^\kappa$, we see that for $m,n \in 2\N$
\[\E \|\sol_{m,n+m}^{u^\kappa,0}\|_{H^1 \to H^{-1}} =\E \|\sol_{0,n}^{u^\kappa,0}\|_{H^1 \to H^{-1}} \leq \bar D_1 e^{-\gamma_1 n}.\]
Then define
  \begin{equation}
  K \defeq \sum_{m \in 2\N}
  \sum_{\substack{n \in 2\N,\\
    n\geq m+4 \frac{\log (m+2)}{\gamma_1} }}
    e^{\gamma_1(n-m)/2}\|\sol^{u_\kappa,0}_{m,n}\|_{H^1 \to H^{-1}}
    \,,
  \end{equation}
and note
  \begin{equation}\label{e:EK}
  \E
  K \leq \sum_{m \in 2\N} \sum_{\substack{ n \in 2\N\,,\\
    n\geq m+4 \frac{\log (m+2)}{\gamma_1}} }
    e^{-\gamma_1(n-m)/2} \leq C \sum_{m \in 2\N} \frac{1}{m^2} \leq C
    < \infty
\,.
  \end{equation}

  Now, for any $m,n \in 2\N$ such that $0 \leq m \leq n$ and $n \geq m + 4 \gamma_1^{-1} \log(m+2)$, we must have
\[\|\sol_{m,n}^{u^\kappa,0}\|_{H^1 \to H^{-1}} \leq Ke^{- \gamma_1 (n-m)/2} .\]
On the other hand if $m \leq n \leq m+4 \gamma_1^{-1} \log(m+2)$, then 
\[\|\sol_{m,n}^{u^\kappa,0}\|_{H^1 \to H^{-1}} \leq 1 \leq e^{2 \log(m+2)} e^{-\gamma_1(n-m)/2} \leq (m+2)^2 e^{-\gamma_1(n-m)/2}.\]
Thus in any case, we have that for any $m,n \in 2\N, 0 \leq m \leq n,$
\[\|\sol_{m,n}^{u^\kappa,0}\|_{H^1 \to H^{-1}} \leq (K + (m+2)^2) e^{- \gamma_1(n-m)/2}.\]
  Replacing~$\gamma_1$ with~$\gamma_1/2$, using~\eqref{e:EK} and Chebychev's inequality will show~\eqref{e:ukappa-mixrateR} for some~$\kappa$-independent constant~$D$.
  \smallskip

\step
  Finally, to estimate $\norm{u^\kappa}_{W^{-1,\infty}}$ we note that for~$t \in [0, 1)$ we have
\begin{equation}
  u^\kappa_t = \grad^\perp H
  \quad\text{where}\quad
  H(x) = A\sum_{i=0}^{N_\kappa-1} \int_0^{x_2} \varphi_\kappa \paren[\Big]{y-\frac{i \pi}{N_\kappa}-\alpha_i^{0}}
  \, dy
  \,.
\end{equation}
Since the support of at most three terms in the sum intersect, we have
\begin{equation}
  \norm{u^\kappa_t}_{W^{-1, \infty}} \leq \norm{H}_{L^\infty}
    \leq 3 A \norm{\varphi_\kappa}_{L^1}
      \overset{\eqref{e:phiKappaDef}}{\leq} 3 A \kappa \norm{\varphi}_{L^1}
  .
\end{equation}
The same bound is true for~$t \geq 1$.
Using this in Proposition~\ref{p:tdis-lower} immediately implies~\eqref{e:tdis-lowerR}.
\end{proof}

For completeness, we also include a proof of Proposition~\ref{p:mixInhomog}.

\begin{proof}[Proof of Proposition~\ref{p:mixInhomog}]
  Lemma~3.3 in~\cite{CoopermanIyerEA25} shows that~\eqref{e:expdecay} holds with~$s = 0$.
  Now the same argument used in the proof of Theorem~\ref{t:ediss-slowR} will show~\eqref{e:expdecay} for all~$s \geq 0$.

  The enhanced dissipation estimate~\eqref{e:edissInhomog} is proved similarly.
  Indeed Lemma~3.3 in fact shows that for any~$q < \infty$ and sufficiently small~$\kappa \geq 0$, the flow of the SDE~\eqref{e:SDEX} is almost surely exponentially mixing with rate~$h(0, t) = D_{\kappa, q} e^{-\gamma t}$.
  Here~$D_{\kappa,q}$ depends on~$\kappa$, but has a~$q$-th moment that is bounded uniformly in~$\kappa$.
  Now the Borel--Cantelli like argument used in the proof of Theorem~\ref{t:ediss-slowR} can be repeated to show the flow of the SDE~\eqref{e:SDEX} is exponentially mixing, with rate
  \begin{equation}
    h(s, t) = C D_{q, \kappa} (1 + s^2) e^{-\gamma t}
    \,.
  \end{equation}
  Following the proof of Theorem~1.1 in~\cite{CoopermanIyerEA25} will yield~\eqref{e:edissInhomog} as desired.
\end{proof}

We prove Lemma~\ref{l:kindd} in Section~\ref{s:emuk} and Proposition~\ref{p:tdis-lower} in Section~\ref{s:slowm}.

\section{Exponential mixing of \texorpdfstring{$u^\kappa$}{ukappa} (Lemma \ref{l:kindd})} \label{s:emuk}

We will prove~$u^\kappa$ is exponentially mixing by verifying the conditions of a Harris theorem.
Once these conditions are verified, we use the Harris theorem to prove geometric ergodicity of the \emph{two point} chain.
Then a result of~\cite{DolgopyatKaloshinEA04} will show that the flows~$u^\kappa$ are \emph{almost surely} exponentially mixing.
This general principle is well-known and has been used in several recent papers~\cite{BedrossianBlumenthalEA21, BlumenthalCotiZelatiEA23, CoopermanIyerEA25}.

The main difficulty in our situation is verifying the conditions of the Harris theorem.
This typically requires finding a small set and a Lyapunov function.
Existing methods~\cite{BedrossianBlumenthalEA21, BlumenthalCotiZelatiEA23} prove the existence of a Lyapunov function \emph{non-constructively} using Furstenberg's criterion.
This can not be used in our situation as it will only prove exponential mixing where the rate constants depend on~$\kappa$ in an uncontrolled manner.
Instead, we will provide an explicit Lyapunov function that is \emph{independent} of~$\kappa$, allowing us to obtain~\eqref{e:ukappa-mixrateR} with~$\kappa$-independent constants~$D, \gamma_1$.

\subsection{Notation and preliminaries}
We define the flow map $\Phi^\kappa: [0, \infty)\times \T^2 \to \T^2$, corresponding to the velocity field $u^\kappa$~\eqref{eq:ukdef}, as the solution to the ODE 
\begin{align*}
  \partial_t\Phi^\kappa_t(x) &= u^\kappa(t,\Phi^\kappa_t(x))\,,\\
  \Phi^\kappa_0(x) &= x\,.
\end{align*}
Notice $(\Phi^\kappa_n)_{n\in \N}$ is a $\T^2$-valued \emph{random dynamical system} (RDS).
Let $(P^\kappa_n)_{n\in \N}$ denote the~$n$-step transition probability
\begin{equation}
  P^\kappa_n(x, A) = \P[\Phi^\kappa_{2n}( x) \in A]\,, \qquad A \in \mathcal B(\T^2)\,.
\end{equation}
The two-point RDS $\Phi^{\kappa, (2)}: [0, \infty) \times \T^{2, (2)} \to \T^{2, (2)}$ is defined by
\begin{equation}
  \Phi^{\kappa, (2)}_t( x,y) = (\Phi^\kappa_t( x), \Phi^\kappa_t( y))\,,
\end{equation}
where $\T^{2, (2)}\defeq (\T^2 \times \T^2) \setminus \Delta$ and $\Delta\defeq \set{(x,y)\in \T^2\st x=y}$.
We use $(P^{\kappa, (2)}_n)_{n\in \N}$ to denote the transition probabilities for the two point chain.
Explicitly,
\begin{equation}
  P^{\kappa,(2)}_{n}((x,y), A) \defeq \P[\Phi^{\kappa, (2)}_{2n}( (x,y))\in A]\,, \qquad A\in \mathcal B(\T^{2, (2)})\,.
\end{equation}

\subsection{The Harris conditions}

We first use a Harris theorem to prove geometric ergodicity of the two point chain, for which we need a Lyapunov function and a small set condition.
As mentioned earlier, the Lyapunov function is usually produced using non-constructive arguments using Furstenberg's criterion, in a manner that the~$\kappa$-dependence is unknown.
We avoid this issue by explicitly producing a~$\kappa$-independent Lyapunov function.
Explicitly, given~$p > 0$ we define the Lyapunov function~$V \colon \T^{2, (2)} \to [1, \infty)$ by
\begin{equation}\label{e:lyapunov}
  V(x,y) \defeq \abs{x-y}_\infty^{-p}
  \quad\text{where}\quad
  \abs{x - y}_\infty
    \defeq \max\set{
      d(x_1, y_1), d( x_2, y_2 )
    }
  \,.
\end{equation}
Here we used the convention that~$x = (x_1, x_2) \in \T^2$, $y = (y_1, y_2) \in \T^2$, and~$d(\cdot, \cdot)$ denotes the distance function on~$\T^1$.
We will now show that~$V$ satisfies a Foster--Lyapunov drift condition with a~$\kappa$-dependent additive constant and a small set condition with a~$\kappa$-dependent minorizing measure.
Nevertheless, we will ensure the multiplicative constants are \emph{independent} of~$\kappa$, which will lead to the proof of Lemma~\ref{l:kindd}.

\begin{lemma}[Harris conditions]\label{l:harrisConditions}
  There exists constants $p\in(0,\frac{1}{12})$,  $\alpha, \gamma_1\in (0,1)$, and $K_1>0$ such that for all sufficiently small~$\kappa > 0$ the following hold.
  \begin{enumerate}\reqnomode
    \item
      The function~$V$ defined by~\eqref{e:lyapunov} satisfies the Foster--Lyapunov drift condition
      \begin{equation}\label{e:lyapunovFoster}
	P^{\kappa,(2)}_{3}V \leq \gamma_1 V+ K_\kappa\,,
	\quad\text{where}\quad
	   K_\kappa \defeq \frac{K_1}{\kappa^{3p}}\,.
    \end{equation}   
    \item 
      There exists~$R_\kappa > 2K_\kappa / (1 - \gamma_1)$, and a probability measure~$\nu_\kappa$ such that
      \begin{equation}\label{eq:smallset}
	\inf_{\set{V \leq R_\kappa}} P^{\kappa,(2)}_{3}(x, \cdot) \geq \alpha  \nu_\kappa(\cdot)
	\,.
      \end{equation}
  \end{enumerate}
\end{lemma}

Lemma~\ref{l:harrisConditions} immediately yields Lemma~\ref{l:kindd} using established methods, and we outline the proof in Section~\ref{s:proofHarris} below.
The bulk of this section is devoted to checking~\eqref{e:lyapunovFoster} and~\eqref{eq:smallset}.

\subsection{The near diagonal Foster--Lyapunov drift condition}

We now prove that the function~$V$ defined in~\eqref{e:lyapunov} satisfies the Foster--Lyapunov drift condition~\eqref{e:lyapunovFoster} near the diagonal.
We will shortly see that this leads to~\eqref{e:lyapunovFoster} everywhere by introducing an additive constant.

\begin{lemma}[Near diagonal drift condition]\label{l:lyapunov}
There exists $\kappa$-independent constants $p\in(0,\frac{1}{12})$, $\gamma\in (0,1)$, and $s_*>0$ such that for all sufficiently small $\kappa$, the function~$V$ defined in~\eqref{e:lyapunov}
satisfies
\begin{equation}\label{eq:lyapunov-condition}
  P^{\kappa,(2)}V \leq \gamma V
  \,,
  \quad \text{on}\quad
  \Delta\paren[\Big]{\frac{s_*}{N_\kappa}}\,.
\end{equation}
  Here, $\Delta(a) \defeq \set{(w,\tilde w)\in \T^{2, (2)}\st \abs{w - \tilde w}_\infty < a}$ is a neighborhood of the diagonal with width~$a$.
\end{lemma}
\begin{proof}
  For~$0 \leq t \leq 1$, the velocity~$u^\kappa$ is a horizontal shear, and so
  \begin{equation}
    (\Phi^\kappa_t)^{-1}(x) = (x_1 - t \psi^\kappa_1(x_2), x_2)
  \end{equation}
  where
  \begin{equation}\label{e:psiDef}
    \psi^\kappa_1(x_2) \defeq A\paren[\Big]{\sum_{i=0}^{2N_\kappa-1} \varphi_\kappa \paren[\Big]{x_2-\alpha_i^{0}}}
    \,,
  \end{equation}
  and~$\alpha_i^0$ are as in~\eqref{eq:ukdef}.
  Hence
  \begin{equation}
    \norm{\grad (\Phi^\kappa_1)^{-1}}_{L^\infty}
      \leq 1 + \norm{\psi^\kappa_1}_{C^1}
      \leq 1 + 3 A N_\kappa \norm{\varphi}_{C^1}
      \,.
  \end{equation}
  The second inequality above is true because for any~$x_2$, at most~$3$ terms in~\eqref{e:psiDef} are nonzero.
  Repeating this once more we obtain
\begin{equation}\label{e:distancelb}
  (10 AN_\kappa \norm{\varphi}_{C^1} )^{-2} \abs{x-y}_\infty \leq \abs{\Phi^\kappa_2(x) - \Phi^\kappa_2(y)}_\infty, \quad \forall (x,y)\in \T^{2, (2)}
\end{equation}
provided $\kappa$ is sufficiently small.
  Combined with~\eqref{e:lyapunov} this implies
  \begin{equation}\label{e:Vstupid}
    V( \Phi^\kappa_2(x), \Phi^\kappa_2(y) )
      \leq C \tilde A_\kappa^{2p} V(x, y)
      \,,
      \quad\text{where}\quad
      \tilde A_\kappa \defeq A N_\kappa
      \,.
  \end{equation}

  Let $s_*$ be a small positive number, that will be chosen later.
  Fix $(x,y)\in \Delta(s_*/N_\kappa)$, and split the proof into the following cases.

\restartcases
\case[$|x_2 - y_2| \geq |x_1 - y_1|$]
Assume that $0 < x_2 < y_2 < \frac{\pi}{N_\kappa}$.
  The other case with $0 < x_2 < \frac{\pi}{N_\kappa} < y_2 < \frac{2\pi}{N_\kappa}$ is similar, the only difference being that we need to use one more independent phase shift.

  We first let $E$ be the event that $\abs{(\Phi^\kappa_1(x) - \Phi^\kappa_1(y)) \cdot e_1} < 2\abs{x_2-y_2}$. We use triangle inequality and divide by $\tilde A_\kappa \abs{x_2-y_2}$ to see that
\begin{equation}
E \subset F\defeq \set[\Bigg]{\abs{f(\zeta_1)+f(\zeta_2)+f(\zeta_3)} < \frac{3}{\tilde A_\kappa}}\,,
\end{equation}
where
  \begin{equation}\label{e:fDef}
    f(s) = \frac{1}{\abs{Nx_2-Ny_2}}(\varphi'*\one_{[-Ny_2, -Nx_2]})(-s) \in C_c^\infty(N_\kappa x_2-2\pi, N_\kappa y_2)\,,
\end{equation}
and~$\zeta_1$, $\zeta_2$, $\zeta_3$ are independent random variables with distribution
  \begin{equation}\label{e:zetaDist}
    \zeta_1 \sim \Unif([-2\pi, -\pi])\,,
    \quad
    \zeta_2 \sim \Unif([-\pi, 0])\,,
    \quad\text{and}\quad
    \zeta_3 \sim \Unif([0, \pi])
    \,.
  \end{equation}

To estimate the probability of the event $F$, define 
\begin{equation}
g(\zeta_1, \zeta_2, \zeta_3) = f(\zeta_1) + f(\zeta_2) + f(\zeta_3)\,.
\end{equation}
  Notice that on the region $F_1=\set{\abs{f'(\zeta_2)} > \paren{3/\tilde A_\kappa}^{1/2}}$, $\partial_{\zeta_2} g = f'(\zeta_2) \neq 0$.
  Thus by the implicit function theorem, there exists a differentiable map $\tilde f$, defined locally on a neighborhood of $(\zeta_1, \zeta_3)$, such that $g(\xi_1, \tilde f (\xi_1, \xi_3), \xi_3)=0$.
  On the surface $\Gamma \defeq \set{ \xi_2 = \tilde f(\xi_1, \xi_3)}$,  we note
\begin{equation}
\norm{\grad g (\xi_1, \xi_2, \xi_3)} = \norm{(f'(\xi_1), f'(\xi_2), f'(\xi_2))} \geq \abs{f'(\xi_2)}> \sqrt{\frac{3}{\tilde A_\kappa}}\,,
\end{equation}
which implies
  \begin{equation}\label{e:PF1}
  \P[F \cap F_1] \leq  \P\bigg[\dist ((\zeta_1, \zeta_2, \zeta_3), \Gamma) \leq \sqrt{\frac{3}{\tilde A_\kappa}}\bigg] \leq \frac{C}{\sqrt{\tilde A_\kappa}}\,.
\end{equation}

  Also, Assumption~(A\ref{i:second-non-degenerate}) implies that for small enough $s_*>0$, 
\begin{equation}\label{e:doublefmin}
\min_{\set{z\st f'(z)=0}} \abs{f''(z)} \geq \frac{1}{2}\min_{\set{z\st \varphi''(z)=0}} \abs{\varphi'''(z)} > 0\,,    
\end{equation}
so that
  \begin{equation}\label{e:PF2}
  \P[F\cap F_1^c] \leq \P[F_1^c] \leq \frac{C}{\sqrt{\tilde A_\kappa}}
  \,.
\end{equation}
Combining~\eqref{e:PF1} and~\eqref{e:PF2} implies
\begin{equation}\label{e:PF}
  \P [F ]\leq \frac{C}{\sqrt{\tilde A_\kappa}}
  \,.
\end{equation}

Consequently,
\begin{align}
\E[V(\Phi^\kappa_2(x), \Phi^\kappa_2(y))] &= \E[V(\Phi^\kappa_2(x), \Phi^\kappa_2(y))(\one_E + \one_{E^c})] \\
&\leq \E[V(\Phi^\kappa_2(x), \Phi^\kappa_2(y))\one_{F}] + 2^{-p} V(x,y)\\
  &\overset{\mathclap{\eqref{e:Vstupid}, \eqref{e:PF}}}{\leq}
    \quad (C\tilde A_\kappa^{2p-\frac{1}{2}} + 2^{-p})V(x,y)\,.
\end{align}
Setting $p\in (0, \frac{1}{4})$ and ensuring~$\kappa$ is small enough proves~\eqref{eq:lyapunov-condition} in this case.

\case[$|x_2 - y_2| < |x_1 - y_1|$]
Again, without loss of generality we may assume that $0< x_2 < y_2 < \frac{\pi}{N_\kappa}$, as we can handle the other case by using one more independent phase shift.

We let $E$ be the event
\begin{equation}
  E \defeq \set[\bigg]{
    \abs{ (\Phi^\kappa_1(x) - \Phi^\kappa_1(y)) \cdot e_1} < \frac{\abs{x_1-y_1}}{\sqrt{\tilde A_\kappa}}
  }
  \,.
\end{equation}
Dividing by $\tilde A_\kappa \abs{x_1-y_1}$ we see~$E \subseteq F$, where
\begin{equation}
  F\defeq \set[\bigg]{ \abs[\Big]{ \frac{1}{\tilde A_\kappa} + c(f(\zeta_1)+f(\zeta_2)+f(\zeta_3))} <
    \frac{1}{\tilde A_\kappa^{3/2} }
  }\,,
  \quad\text{where}\quad
  c \defeq  \frac{\abs{x_2-y_2}}{\abs{x_1-y_1}} < 1
  \,,
\end{equation}
and~$\zeta_1$, $\zeta_2$, $\zeta_3$ are independent uniform random variables distributed according to~\eqref{e:zetaDist}, and~$f$ is the same function in~\eqref{e:fDef}.

Notice that when
\begin{equation}\label{e:cSmall}
  c < \frac{\tilde A_\kappa^{-1}-\tilde A_\kappa^{-\frac{3}{2}}}{3\norm{\varphi'}_\infty}
\end{equation}
we have
\begin{equation}
     \abs[\Big]{ \frac{1}{\tilde A_\kappa} + c(f(\zeta_1)+f(\zeta_2)+f(\zeta_3))}\geq \tilde A_\kappa^{-\frac{3}{2}}\,,
\end{equation}
and hence $\P[F]=0$.
Thus, we may assume~\eqref{e:cSmall} does not hold, and consider the event $F_2 \defeq \set{\abs{f'(\zeta_2)} > \tilde A_\kappa^{-\frac{1}{3}}}$.
Using the same argument as we used to obtain~\eqref{e:PF1} and~\eqref{e:PF2} we obtain
\begin{equation}
    \P[F\cap F_2^c] \leq C\tilde A_\kappa^{-\frac{1}{3}}\,,
    \qquad
    \P[F \cap F_2] \leq \frac{C \tilde A_\kappa^{-\frac{7}{6}} }{c} \leq C\tilde A_\kappa^{-\frac{1}{6}}\,,
\end{equation}
and hence
\begin{equation}\label{e:PEBound}
  \P[E] \leq \P [F] \leq C\tilde A_\kappa^{-\frac{1}{6}}\,.
\end{equation}

Now consider the events 
\begin{align}
  \label{e:G1} G_1 &= \set{\abs{ (\Phi^\kappa_1(x) - \Phi^\kappa_1(y)) \cdot e_1} < \eta\abs{x_1-y_1}}\,,\\
  \label{e:G2} G_2 &= \set{\abs{ (\Phi^\kappa_2(x) - \Phi^\kappa_2(y)) \cdot e_2} < 5\tilde A_\kappa^\frac{1}{2}\abs{ (\Phi^\kappa_1(x) - \Phi^\kappa_1(y)) \cdot e_1}}\,,
\end{align}
for some large $\eta>0$.
On the event $G_1$, provided that $s_*$ is small enough,  there exists~$i \in \N$ such that either
\begin{equation}
  \frac{\pi i}{N_\kappa} < \Phi^\kappa_1(x) \cdot e_1 < \Phi^\kappa_1(y) \cdot e_1  < \frac{\pi(i+1)}{N_\kappa}
\end{equation}
 or
 \begin{equation}
   \frac{\pi i}{N_\kappa} < \Phi^\kappa_1(x) \cdot e_1 < \frac{\pi(i+1)}{N_\kappa} < \Phi^\kappa_1(y) \cdot e_1  < \frac{\pi(i+2)}{N_\kappa}
   \,.
 \end{equation}
  Without loss of generality, we can assume the first case with $i=0$.
  (The second case can be handled similarly by considering one additional,  independent phase shift.)
As before, we use the triangle inequality and divide by~$\tilde A_\kappa \abs{ (\Phi^\kappa_1(x) - \Phi^\kappa_1(y)) \cdot e_1}$ to obtain
\begin{equation}
    E^c \cap G_1 \cap G_2 \subset  \set[\Big]{\abs{h(\zeta_1)+h(\zeta_2)+h(\zeta_3)} < 6\tilde A_\kappa^{-\frac{1}{2}}}\,,
\end{equation}
where
\begin{equation}
  h(s) \defeq \frac{1}{\abs{N_\kappa(\Phi^\kappa_1(x) - \Phi^\kappa_1(y)) \cdot e_1}}(\varphi'*\one_{[-N_\kappa\Phi^\kappa_1(y) \cdot e_1, -N_\kappa\Phi^\kappa_1(x) \cdot e_1]})(-s)\,,
\end{equation}
and~$\zeta_1$, $\zeta_2$, $\zeta_3$ are independent random variables distributed according to~\eqref{e:zetaDist}.
By considering events $\set{\abs{h'(\zeta_2)}> \tilde A_\kappa^{-\frac{1}{3}}}$ and its complement, we get
\begin{equation}\label{e:PEcG1G2}
    \P[E^c \cap G_1 \cap G_2] \leq C\tilde A_\kappa^{-\frac{1}{6}}\,.
\end{equation}

Using~\eqref{e:PEBound}, \eqref{e:G1}, \eqref{e:G2} and~\eqref{e:PEcG1G2} we see
\begin{align}
\E[V(\Phi^\kappa_2&(x), \Phi^\kappa_2(y))] \\
  &=\E[V(\Phi^\kappa_2(x), \Phi^\kappa_2(y))(\one_E +  \one_{E^c\cap G_1 \cap G_2} + \one_{E^c\cap G_1^c} + \one_{E^c\cap G_1 \cap G_2^c} )] \\
&\leq (C\tilde A_\kappa^{2p-\frac{1}{6}} + C\tilde A_\kappa^{2p-\frac{1}{6}} + \eta^{-p} + 5^{-p})V(x,y)\,.
\end{align}
Choosing $p\in (0, \frac{1}{12})$, $\eta$ large, and $\kappa$ small proves~\eqref{eq:lyapunov-condition} and concludes the proof.
\end{proof}

\subsection{Verifying the Harris conditions (Lemma \ref{l:harrisConditions})}
To prove Lemma~\ref{l:harrisConditions}, we first note that the near-diagonal drift condition (Lemma~\ref{l:lyapunov}) implies the drift condition~\eqref{e:lyapunovFoster} everywhere.
To complete the proof of Lemma~\ref{l:harrisConditions}, we need to prove that the sublevel set $S \coloneq \{V \leq C K_\kappa\}$ is small for a large enough constant $C > 0$.
By construction, elements of~$S$ are~$O(1/N_\kappa^3)$ away from the diagonal in~$\T^{2, (2)}$.
But for most pairs of nearby points, the map~$\Phi_1$ expands distances by a factor of~$N_\kappa$.
Thus, after two iterations, elements of~$S$ will move to a distance of~$1/N_\kappa$ away from the diagonal (Lemma~\ref{l:push-to-off-diagonal}).

When two points in~$\T^{2, (2)}$ are distance at least~$1/N_\kappa$ apart, there is a positive $\kappa$-independent probability event that the two points are hit by independent shears.
When the amplitude of the shear is large enough, the law of the points after one more iteration dominates a small multiple of Lebesgue measure on~$\T^{2,(2)}$ away from the diagonal.

\begin{proof}[Proof of Lemma~\ref{l:harrisConditions}]
Let $\gamma, p$, and $s_*$ be the constants as in Lemma~\ref{l:lyapunov}. Then, for some $\kappa$-independent constant $K_1' > 0$ we have
  \begin{equation}
    V( \Phi^\kappa_2(x), \Phi^\kappa_2(y)) \overset{\eqref{e:Vstupid}}{\leq} K_1' N_\kappa^{3p}
    \quad\text{for}\quad
    (x,y)\in \Delta \paren[\Big]{\frac{s_*}{N_\kappa}}^c\,.
  \end{equation}
Combining this with~\eqref{eq:lyapunov-condition} , we have
\begin{equation}
  P^{\kappa, (2)}V \leq \gamma V + K_\kappa' 
  \,,
  \quad\text{where}\quad \label{e:Kk'def}
  K_\kappa' \defeq K_1' N_\kappa^{3p}
  \,.
\end{equation}
This implies for any $l\geq 0$,
\begin{equation}
  P^{\kappa,(2)}_{l} V \leq \gamma^l V +  K_\kappa' \frac{1-\gamma^l}{1-\gamma}\,.
\end{equation}
  Choosing~$l = 3$, this proves~\eqref{e:lyapunovFoster} with~$\gamma_1 = \gamma^3$, and~$K_1 = K_1' (1 - \gamma^3)/(1 - \gamma)$.

  The rest of this proof is devoted to checking the small set condition~\eqref{eq:smallset}.
  For this, we will find
  \begin{equation}\label{e:Rkappa}
    R_\kappa>\frac{2K_\kappa'}{1-\gamma}
    =\frac{2 K_\kappa}{1-\gamma_1}
    \,.
  \end{equation}
  and show that $\set{V \leq R_\kappa}$ is indeed a small set for $P^{\kappa,(2)}_3$ as in~\eqref{eq:smallset}.

  Notice that for any small~$\kappa$-independent constant~$\eta > 0$, and~$R_\kappa$ defined by
  \begin{equation}
    R_\kappa \defeq \frac{N_\kappa^{3p}}{\eta^p}
      = \paren[\Big]{\frac{1 - \gamma}{2\eta^p K_1'}} \frac{2 K'_\kappa }{1 - \gamma}
    \,,
  \end{equation}
  $R_\kappa$ satisfies the condition~\eqref{e:Rkappa} and we have
  \begin{equation}
    S \defeq \set{V \leq R_\kappa}
      = \set[\Big]{(x,y)\in \T^{2, (2)}\st \abs{x-y}_\infty \geq \frac{\eta}{N_\kappa^3}}
    \,.
  \end{equation}

Now, we define
\begin{align}
    \label{e:s1def}
    S_1 &\defeq \set{(x,y)\in S \st \abs{x_2-y_2}\geq \frac{\eta}{N_\kappa^3}}\,,\\
    \label{e:s2def}
    S_2 &\defeq S - S_1 = \set{(x,y)\in S \st \abs{x_1-y_1} \geq \frac{\eta}{N_\kappa^3} > \abs{x_2-y_2}}\,,\\
    \label{e:s3def}
    S_3 &\defeq \set{(x,y)\in S \st \abs{x_2-y_2}\geq \frac{\eta}{N_\kappa}}\,.
\end{align}

Then we prove
\begin{align}
\label{e:oneshift} \inf_{(x,y)\in S_3}
\Law ((\Phi^\kappa_1(x) \cdot e_1, \Phi^\kappa_1(y) \cdot e_1)) &\geq C(\eta, \varphi) \Leb_{\T^2}\\
\label{e:infS1} c_1\defeq \inf_{z\in S_1} P^{\kappa,(2)}(z, S_3) &> 0\,,\\
\label{e:infS2} c_2\defeq \inf_{z\in S_2} P_2^{\kappa,(2)}(z, S_3) &> 0\,.
\end{align}

Indeed, suppose that we proved~\eqref{e:oneshift}--\eqref{e:infS2}. 
Then, we can first prove that  
\begin{equation}\label{e:infS3}
     \inf_{z\in S_3} P^{\kappa,(2)}(z, \cdot) \geq C(\eta)\nu_\kappa(\cdot)\,, \\
\end{equation}
where
\begin{equation}
    \nu_\kappa \defeq \Leb\bigr|_{\Delta_{1}(\frac{\eta}{N_\kappa})^c}\,, \quad\text{and}\quad 
    \Delta_1(a) \defeq \set{(w, \tilde w) \in \T^{2, (2)}\st \abs{w_1 - \tilde w_1} < a}\,.
\end{equation}
This is because for any open sets $U_1, U_2, V_1, V_2 \subseteq \T^1$, if we let
\begin{align}
  E_1 &\defeq \set{\Phi^\kappa_2(x) \cdot e_2 \in U_2} \cap \set{\Phi^\kappa_2(y) \cdot e_2)\in V_2}\,,\\
    E_2 &\defeq \set{\Phi^\kappa_1(x) \cdot e_1 \in U_1} \cap \set{\Phi^\kappa_1(y) \cdot e_1 \in V_1}\,,\\
    E_3 &\defeq \set[\Big]{\abs{ (\Phi^\kappa_1(x) - \Phi^\kappa_1(y)) \cdot e_1}>\frac{\eta}{N_\kappa}}\,,
\end{align}
then on the event $E_3$, the vertical shear case is similar to~\eqref{e:oneshift} so that
\begin{align}
    \P[E_1\cap E_2] &\geq \P[E_1 \st E_2\cap E_3]\P[E_2\cap E_3]
    \geq C(\eta) \abs{U_2 \times V_2} \abs{W}
    \,,
\end{align} 
where
\begin{equation}
  W \defeq \set[\Big]{ (w_1, \tilde w_1) \in U_1 \times V_1 \st \abs{w_1-\tilde w_1}>\frac{\eta}{N_\kappa}}\,.
\end{equation}
This implies
\begin{equation}
  \P [E_1 \cap E_2] \geq C(\eta) \abs[\Big]{ U_1 \times U_2 \times V_1 \times V_2 - \Delta_1\paren[\Big]{\frac{\eta}{N_\kappa}} }
  \,,
\end{equation}
which implies~\eqref{e:infS3} as claimed.

Now, to prove~\eqref{eq:smallset}, we notice that $S_3 \subset S_1$ and hence for any $z\in S_1$, 
\begin{equation}
    P_2^{\kappa,(2)}(z, S_3) \geq \int_{S_3} P^{\kappa,(2)}(y, S_3) P^{\kappa,(2)}(z, dy) \overset{\eqref{e:infS1}}{\geq} c_1^2\,.
\end{equation}
Combining this with~\eqref{e:infS2}, we get
\begin{equation}
    \inf_{z\in S} P_2^{\kappa,(2)}(z, S_3) \geq \min\set{c_1^2, c_2}\,,
\end{equation}
and using~\eqref{e:infS3} yields that for any $z\in S$, and~$B \subseteq \T^{2, (2)}$ we have
\begin{equation}
    P_3^{\kappa, (2)}(z, B) \geq \int_{S_3} P_2^{\kappa, (2)}(z, dy) P^{\kappa, (2)}_1(y, B)  \geq C(c_1,c_2,\eta)\nu_\kappa(B)\,,
\end{equation}
as desired.

It remains to prove~\eqref{e:oneshift}--\eqref{e:infS2}, which we do in Lemmas~\ref{l:s3-small}, \ref{l:push-to-off-diagonal} and~\ref{l:push-to-off-diagonal2}, below.
\end{proof}

To prove~\eqref{e:oneshift}--\eqref{e:infS2}, we need an elementary fact about the pushforward of a uniform random variable by the function~$\varphi$.
\begin{lemma}\label{l:density-lb}
  For all sufficiently small $\eta>0$, and $a \in [0, 2\pi -\frac{\eta}{2}]$, we have
    \begin{equation}
        \Law (\varphi(U_{a, \eta})) \geq C(\varphi')\Leb|_{[\varphi(a), \varphi(a)+C(\eta, \varphi'')]}\,,
    \end{equation}
    for some $\kappa$-independent constants $C(\varphi')$ and $C(\eta, \varphi'')$.
    Here $U_{a, \eta}\sim \Unif([a, a+\frac{\eta}{2}])$.
\end{lemma}
\begin{proof}
    Assumption~(A\ref{i:first-non-degenerate}) implies that for sufficiently small $\eta>0$,
    \begin{equation}
        \frac{d}{d\Leb}\Law (\varphi(U_{a, \eta})) \geq \frac{1}{\norm{\varphi'}_\infty}\one_{[\varphi(a), \varphi(a+\frac{\eta}{2})]}\,.
    \end{equation}

    By Taylor's theorem we know
    \begin{equation}
      \abs[\Big]{
	\varphi\paren[\Big]{ a + \frac{\eta}{2} } - \varphi(a)
      }
      \geq \max\set[\Big]{
	\frac{\abs{\varphi'(a)} \eta}{4},~
	\frac{\abs{\varphi''(a)} \eta^2 }{16} - \frac{\abs{\varphi'(a)} \eta }{2}
      }\,.
    \end{equation}
    By Assumption~(A\ref{i:first-non-degenerate}),
    \begin{equation}
      \inf_{a \in [0, 2\pi - \frac{\eta}{2}]}
	\max\set[\Big]{
	  \frac{\abs{\varphi'(a)} \eta}{4},~
	  \frac{\abs{\varphi''(a)} \eta^2 }{16} - \frac{\abs{\varphi'(a)} \eta }{2}
	}
	= C(\eta) > 0
	\,,
    \end{equation}
    concluding the proof.
\end{proof}

We now prove~\eqref{e:oneshift} holds.
\begin{lemma}\label{l:s3-small}
    Let $S_3$ be defined as in~\eqref{e:s3def}.
    Then \eqref{e:oneshift} holds for all sufficiently small $\eta>0$.
\end{lemma}
    
\begin{proof}
We fix $(x,y) \in S_3$ and split the proof into three cases.
\restartcases
\case[$0< x_2 < y_2 < \frac{\pi}{N_\kappa}$]
  There are at most three terms in the sum in~\eqref{eq:ukdef} which may not vanish at both~$x_2$ and~$y_2$.
  These terms are the ones with index~$i = 2N_\kappa - 2$, $i = 2N_\kappa -1$, and~$i = 0$.
  For simplicity of notation, we will use~$\alpha_{-2}$, $\alpha_{-1}$ to denote~$\alpha_{2N_\kappa - 2}\in [-\frac{2\pi}{N_\kappa}, -\frac{\pi}{N_\kappa}]$ and~$\alpha_{2N_\kappa - 1}\in [-\frac{\pi}{N_\kappa}, 0]$ respectively.

  On the event $E=\set{\alpha_{-2}\in [x_2-\frac{2\pi}{N_\kappa}, y_2-\frac{2\pi}{N_\kappa}]}\cap \set{\alpha_{0}\in [x_2, y_2]}$, only terms in the sum in~\eqref{eq:ukdef} with index~$2N_\kappa - 2$ and~$2N_\kappa -1$ may not vanish at~$x_2$,
  and only terms in the sum in~\eqref{eq:ukdef} with index~$2N_\kappa - 1$ and~$0$ may not vanish at~$y_2$.
  Thus, after fixing $\alpha_{-1}$, the random variables $\Phi^\kappa_1(x) \cdot e_1$ and $\Phi^\kappa_1(y) \cdot e_1$ satisfy 
\begin{equation}
\Law ((\Phi^\kappa_1(x) \cdot e_1, \Phi^\kappa_1(y) \cdot e_1)\st E, \alpha_{-1}) = 
\begin{pmatrix}
x_1+\varphi_\kappa(x_2-\alpha_{-1}) \\
y_1+\varphi_\kappa(y_2-\alpha_{-1})
\end{pmatrix}
+
\begin{pmatrix}
\Law (A\varphi(U))\\
\Law (A\varphi(V))
\end{pmatrix}\,,
\end{equation}
where 
$U, V$ are independent random variables with distributions
\begin{equation}
 U\sim \Unif([2\pi-N_\kappa(y_2-x_2), 2\pi])\,, \quad V\sim \Unif([0, N_\kappa(y_2-x_2)])\,.   
\end{equation}
Also, 
\begin{equation}
    \P[E] = \paren[\Big]{\frac{N_\kappa(y_2-x_2)}{\pi}}^2 \geq \paren[\Big]{\frac{\eta}{\pi}}^2\,,
\end{equation}
and the supports of the distributions of $U$ and $V$ have Lebesgue measure of at least $\eta>0$. 

Then Lemma~\ref{l:density-lb} implies that for $A$ sufficiently large, independent of $x$ and $y$,
\begin{equation}\label{eq:firstsmall}
    \Law ((\Phi^\kappa_1(x) \cdot e_1, \Phi^\kappa_1(y) \cdot e_1)) \geq C(\eta, \varphi) \paren[\bigg]{\frac{\eta}{\pi}}^2 \Leb_{\T^2}\,.
\end{equation}

\case[$0 < x_2 < \frac{\pi}{N_\kappa} < y_2 < \frac{2\pi}{N_\kappa}$]
As in the first case, only the terms in the sum in~\eqref{eq:ukdef} with index $i=2N_\kappa-2, i=2N_\kappa-1$, and $i=0$ may not vanish at $x_2$ and only the terms with index $i=2N_\kappa-1, i=0$, and $i=1$ may not vanish at $y_2$. 
\smallskip

Define the events
  \begin{align}
    E_1 &\defeq \set[\Big]{\alpha_{-1}\in \brak[\Big]{-\frac{\pi}{N_\kappa}, y_2-\frac{2\pi}{N_\kappa}}}\cap \set[\Big]{\alpha_1 \in \brak[\Big]{\frac{\pi}{N_\kappa}, y_2}}
    \,,
    \\
    E_2 &\defeq \set[\Big]{\alpha_{-2}\in \brak[\Big]{x_2-\frac{2\pi}{N_\kappa}, -\frac{\pi}{N_\kappa}}}\cap \set[\Big]{\alpha_0 \in \brak[\Big]{x_2, \frac{\pi}{N_\kappa}}}
    \,.
  \end{align}
On the event~$E_1$, only the terms in the sum in~\eqref{eq:ukdef} with index $i=0$ and $i=1$ may not vanish at $y_2$. Thus, 
\begin{align}
\Law ((\Phi^\kappa_1(x)_1, \Phi^\kappa_1(y)_1)&\st E_1, \alpha_{-2}, \alpha_0) = \\
&\begin{pmatrix}
x_1+\varphi_\kappa(x_2-\alpha_{-2})+\varphi_\kappa(x_2-\alpha_0) \\
y_1+\varphi_\kappa(y_2-\alpha_0)
\end{pmatrix}
+
\begin{pmatrix}
\Law (A\varphi(U))\\
\Law (A\varphi(V))
\end{pmatrix}\,,
\end{align}
where 
$U, V$ are independent random variables with distribution
\begin{equation}
U\sim \Unif([2\pi-N_\kappa(y_2-x_2), \pi+N_\kappa x_2])\,, \quad V\sim \Unif([0, N_\kappa y_2-\pi])\,.
\end{equation}
 
We note that
\begin{align}
    \P[E_1] &= \paren[\Big]{\frac{N_\kappa y_2}{\pi}-1}^2\,, \\
    \P[E_2] &= \paren[\Big]{1-\frac{N_\kappa x_2}{\pi}}^2\,, \\
    \P[E_1] + \P[E_2] &\geq 2 \paren[\Big]{\frac{N_\kappa(y_2-x_2)}{2\pi}}^2 > 2\paren[\Big]{\frac{\eta}{2\pi}}^2\,.
\end{align}

Without loss of generality, if we assume that $\P[E_1] \geq \P[E_2]$, then $\P[E_1]\geq\paren{\frac{\eta}{2\pi}}^2$ so that the support of the distributions of $U$ and $V$ each has Lebesgue measure of at least $\frac{\eta}{2}$. Combining this with Lemma~\ref{l:density-lb} implies that if $A$ is sufficiently large, independent of $x$ and $y$, then 
\begin{align}\label{eq:secondsmall}
    \Law ((\Phi^\kappa_1(x) \cdot e_1, \Phi^\kappa_1(y) \cdot e_1)) &\geq \Law ((\Phi^\kappa_1(x) \cdot e_1, \Phi^\kappa_1(y) \cdot e_1)\st E_1)\P[E_1] \\
    &\geq C(\eta, \varphi) \paren[\Big]{\frac{\eta}{2\pi}}^2 \Leb_{\T^2}\,.
\end{align}

The case $\P[E_2] \geq \P[E_1]$ is similar and can be proved by using the event $E_2$ instead.

\case[$0 < x_2 < \frac{\pi}{N_\kappa} < \frac{2\pi}{N_\kappa} < y_2$]
In this case the only terms in the sum in~\eqref{eq:ukdef} that may not vanish at~$x_2$ are the ones with index $i=2N_\kappa-2$, $i=2N_\kappa-1$, and $i=0$.
Also, the only terms that may not vanish at $y_2$ are the ones with with index $i=j-1$, $i=j$, and $i=j+1$ where~$j = \floor{\frac{Ny}{\pi}}-1$.
Since~$2 N_\kappa -1 \neq j$, we can write
\begin{multline}
\Law((\Phi^\kappa_1(x)_1, \Phi^\kappa_1(y)_1) \st \alpha_{-2}, \alpha_0, \alpha_{j-1}, \alpha_{j+1}) =\\ 
\begin{pmatrix}
x_1+\varphi_\kappa(x_2-\alpha_{-2}) + \varphi_\kappa(x_2-\alpha_{0}) \\
y_1+\varphi_\kappa(y_2-\alpha_{j-1}) +\varphi_\kappa(y_2-\alpha_{j+1})
\end{pmatrix}
+
\begin{pmatrix}
\Law (A\varphi(U))\\
\Law (A\varphi(V))
\end{pmatrix}\,,
\end{multline}
where $U, V$ are independent random variables such that
\begin{equation}
U\sim \Unif([N_\kappa x_2, \pi+N_\kappa x_2])\,, \quad V\sim \Unif([N_\kappa y_2 - \pi(j+1), N_\kappa y_2-\pi j])\,.   
\end{equation}
 Hence for $A$ sufficiently large (independent of $x$ and $y$),
\begin{equation}\label{eq:thirdsmall}
    \Law ((\Phi^\kappa_1(x) \cdot e_1, \Phi^\kappa_1(y) \cdot e_1)) \geq C(\varphi) \Leb_{\T^2}\,.
\end{equation}

Thus, \eqref{eq:firstsmall}, \eqref{eq:secondsmall}, and~\eqref{eq:thirdsmall} show that~\eqref{e:oneshift} holds in each case, provided $\abs{x_2-y_2} > \frac{\eta}{N_\kappa}$.
This concludes the proof for~\eqref{e:oneshift}.
\end{proof}

\begin{lemma}\label{l:push-to-off-diagonal}
    Let $S_1, S_2$ and $S_3$ be defined as in~\eqref{e:s1def}--\eqref{e:s3def}.
    Then~\eqref{e:infS1} holds for all sufficiently small $\eta>0$.
\end{lemma}
\begin{proof}
Fix $(x,y) \in S_1$ and let $\epsilon>0$ be a small $\kappa$-independent constant that will be chosen later. 

\restartcases
\case[$\frac{\epsilon}{N_\kappa^2} < \abs{x_2-y_2} < \frac{\eta}{N_\kappa}$]\label{c:bigx2-y2}
Without loss of generality, we assume that $0 < x_2 < y_2 < \frac{\pi}{N_\kappa}$.
We define the event 
\begin{equation}
    F\defeq \set[\Big]{\abs{(\Phi_1^\kappa(x) - \Phi_1^\kappa(y))\cdot e_1} < N_\kappa\abs{x_2-y_2}}\,.
\end{equation}
Dividing the inequality by $AN_\kappa\abs{x_2-y_2}$ yields 
\begin{equation}
  F \subset F_1 \defeq \set[\bigg]{ \abs[\Big]{ \frac{x_1-y_1}{AN_\kappa \abs{x_2-y_2}} + (f(\zeta_1)+f(\zeta_2)+f(\zeta_3))} <
    \frac{1}{A}
  }\,,
\end{equation}
where $f$ and $\zeta_i$'s are defined as in~\eqref{e:fDef} and~\eqref{e:zetaDist}, respectively. 

For sufficiently small $\eta$>0, \eqref{e:doublefmin} holds and the arguments similar to those used to derive~\eqref{e:PF} yield 
\begin{equation}
    \P[F_1] \leq CA^{-\frac{1}{2}}\,.
\end{equation}
Thus we can ensure $\P[F^c] \geq \frac{1}{2}$ for all sufficiently large $A$.

We notice that on the event $F^c$, 
\begin{equation}
    \abs{(\Phi_1^\kappa(x) - \Phi_1^\kappa(y))\cdot e_1} > N_\kappa\abs{x_2-y_2} > \frac{\epsilon}{N_\kappa}\,.
\end{equation}
Combining this with the symmetric vertical-shear version of~\eqref{e:oneshift}, we get
\begin{equation}
  P^{\kappa, (2)}((x,y), S_3) \geq \frac{1}{2} C(\epsilon, \varphi)\Leb(S_3)\,.
\end{equation}
Setting $\kappa$ small enough yields~\eqref{e:infS1}.

\case[$\frac{\eta}{N_\kappa^3} < \abs{x_2-y_2} < \frac{\epsilon}{N_\kappa^2}$]
Again, without loss of generality, we assume that $0 < x_2 < y_2 < \frac{\pi}{N_\kappa}$.
Suppose $\abs{x_1-y_1} \geq \frac{4A\epsilon}{N_\kappa}$. Then by triangle inequality,
\begin{equation}
    \abs{(\Phi_1^\kappa(x)-\Phi_1^\kappa(y))\cdot e_1} \geq \frac{4A\epsilon}{N_\kappa}- 3AN_\kappa\abs{x_2-y_2} \geq \frac{A\epsilon}{N_\kappa}\,.  
\end{equation}
Then for sufficiently small $A\epsilon$, the vertical-shear case symmetric to~\eqref{e:oneshift} implies
\begin{equation}
  P^{\kappa, (2)}((x,y), S_3) \geq C(A\epsilon, \varphi)\Leb(S_3)\,.
\end{equation}

For the other case $\abs{x_1-y_1} < \frac{4A\epsilon}{N_\kappa}$, we see from triangle inequality that
\begin{equation}
    \abs{(\Phi_1^\kappa(x)-\Phi_1^\kappa(y))\cdot e_1} \leq \frac{4A\epsilon}{N_\kappa} + 3AN_\kappa\abs{x_2-y_2} \leq \frac{7A\epsilon}{N_\kappa}\,.
\end{equation}
Define
\begin{align}
    F_1 &=\set{\abs{(\Phi_1^\kappa(x)-\Phi_1^\kappa(y))\cdot e_1} < N_\kappa \abs{x_2-y_2}}\,, \\
    F_2 &=\set{\abs{(\Phi_2^\kappa(x)-\Phi_2^\kappa(y))\cdot e_2} < N_\kappa \abs{(\Phi_1^\kappa(x)-\Phi_1^\kappa(y))\cdot e_1}}\,.
\end{align}
As in the previous case, we can prove that when $A \epsilon$ is sufficiently small we have
\begin{align}
    \P[F_1^c] \geq \frac{1}{2} \quad\text{and}\quad \P[F_2^c \st F_1^c] \geq \frac{1}{2}\,.
\end{align}
We notice that on the event $F_1^c \cap F_2^c$,
\begin{equation}
    \abs{(\Phi_2^\kappa(x)-\Phi_2^\kappa(y))\cdot e_2} > N_\kappa \abs{(\Phi_1^\kappa(x)-\Phi_1^\kappa(y))\cdot e_1} 
    > N_\kappa^2 \abs{x_2-y_2} > \frac{\eta}{N_\kappa}\,,
\end{equation}
and hence 
\begin{equation}
    P^{\kappa, (2)}((x,y),S_3)
      \geq \P [F_1^c \cap F_2^c ]
      \geq \frac{1}{4}\,.
\end{equation}

\case[$\abs{x_2-y_2} > \frac{\eta}{N_\kappa}$]
We define the event 
\begin{equation}
    G \defeq
      \set[\Big]{\abs{(\Phi_1^\kappa(x)-\Phi_1^\kappa(y))\cdot e_1} \geq  \frac{\eta}{N_\kappa}}
      = \set[\Big]{(\Phi_1^\kappa(x), \Phi_1^\kappa(y)) \notin \Delta_1\paren[\Big]{\frac{\eta}{N_\kappa}} }
      \,,
\end{equation}
and use~\eqref{e:oneshift} to see that
\begin{equation}
  \P[G] \geq \delta \defeq C(\eta, \varphi)\Leb\paren[\Big]{\Delta_1\paren[\Big]{\frac{\eta}{N_\kappa}}^c }\,.
\end{equation}
Conditioned on the event $G$, we use the vertical-shear case symmetric to~\eqref{e:oneshift} to see that
\begin{equation}
    \P[(\Phi_2^{\kappa}(x), \Phi_2^{\kappa}(y))\in S_3] \geq \delta \P[G]  \geq \delta^2\,. 
 \end{equation}
In conclusion, setting $\epsilon=1/A^2$, $A$ large, and $\kappa$ small enough yields~\eqref{e:infS1}.
\end{proof}

\begin{lemma}\label{l:push-to-off-diagonal2}
    Let $S_1, S_2$ and $S_3$ be defined as in~\eqref{e:s1def}--\eqref{e:s3def}.
    Then~\eqref{e:infS2} holds for all sufficiently small $\eta>0$.
\end{lemma}
\begin{proof}
We fix $(x,y) \in S_2$.
By~\eqref{e:Kk'def}, we have that
\begin{align}
  \P[\abs{\Phi^\kappa_2(x)-\Phi^\kappa_2(y)}_\infty \leq \frac{\eta}{N_\kappa^3}] &= \P\brak[\Big]{V(\Phi^\kappa_2(x),\Phi^\kappa_2(y)) \geq \paren[\Big]{\frac{\eta}{N_\kappa^3}}^{-p}}\\
    &\leq \paren[\Big]{\frac{\eta}{N_\kappa^3}}^p(\gamma V(x,y) + K_\kappa')\\
    &\leq \gamma + K_1'\eta^p\,,  
\end{align}
Thus, for all sufficiently small $\eta>0$,
\begin{align}\label{e:e1+e2}
  \P\brak[\Big]{\abs{\Phi^\kappa_2(x)-\Phi^\kappa_2(y)}_\infty > \frac{\eta}{N_\kappa^3}} \geq C_\gamma > 0\,.
\end{align}

Now, we let
\begin{align}
    E &\defeq \set{\abs{(\Phi^\kappa_4(x)-\Phi^\kappa_4(y)) \cdot e_2} > \frac{\eta}{N_\kappa}}\,,\\
  E_1 &\defeq \set{\abs{(\Phi^\kappa_2(x)-\Phi^\kappa_2(y)) \cdot e_2} > \frac{\eta}{N_\kappa^3}}\,,\\
  E_2 &\defeq \set{\abs{(\Phi^\kappa_1(x)-\Phi^\kappa_1(y))\cdot e_1} > \frac{\eta}{N_\kappa^3}}\,,\\
  E_3 &\defeq \set{\abs{(\Phi^\kappa_3(x)-\Phi^\kappa_3(y))\cdot e_1} > \frac{\eta}{N_\kappa}}\,.
\end{align}
Conditioned on $E_1-E_2$, using~\eqref{e:infS1} yields
\begin{equation}\label{e:e1-e2}
    \P[E\st E_1-E_2] \geq c_1 \,.
\end{equation}

Repeating the argument used to derive~\eqref{e:infS1} yields
\begin{equation}\label{e:e2condition}
    \P[E_3\st E_2] \geq c_1\,,
\end{equation}
and using the vertical shear case symmetric to~\eqref{e:oneshift} yields
\begin{equation}\label{e:e3condition}
    \P[E \st E_2\cap E_3] \geq C(\eta, \varphi)\Leb(S_3)\,.
\end{equation}

Combining~\eqref{e:e1-e2}, \eqref{e:e2condition}, \eqref{e:e3condition}, and~\eqref{e:e1+e2}, we have that for some small $\kappa$-independent constant $c'$ and sufficiently small $\kappa$,
\begin{align}
    \P[E] &\geq \P[E\st E_1-E_2]\P[E_1-E_2] + \P[E\st E_2 \cap E_3] \P[E_3\st E_2] \P[E_2]\\
    &\geq c_1 \P[E_1-E_2]+ c'\P[E_2]\\
    &\geq \min\set{c_1, c'} \P[E_1 \cup E_2]\\
    &\geq C(\eta, \gamma)\,.
\end{align}
This concludes the proof for~\eqref{e:infS2}.
\end{proof}

\subsection{Completing the proof of exponential mixing (Lemma \ref{l:kindd})}\label{s:proofHarris}
We will now use Lemma~\ref{l:harrisConditions} to prove Lemma~\ref{l:kindd}.
The proof follows established methods with only minor changes due to the~$\kappa$ dependent constant in~\eqref{e:lyapunovFoster}.
For brevity, we follow the proof of Lemma~3.3 in~\cite{CoopermanIyerEA25} and only point out where changes need to be made.

\begin{proof}[Proof of Lemma~\ref{l:kindd}]
  Following the proof of Lemma~3.2 in~\cite{CoopermanIyerEA25}, will give the~$\kappa$-dependent geometric ergodicity bound
  \begin{equation}
    \norm[\Big]{P^{\kappa,(2)}_{n}f - \int f d\pi^{(2)}}_{\beta_\kappa} \leq C \exp(-\gamma n)\norm[\Big]{f - \int f d\pi^{(2)}}_{\beta_\kappa}
  \,,
  \end{equation}
  for any test function~$f$.
  Here $\beta_\kappa$ and~$\norm{\cdot}_{\beta_\kappa}$ are defined by
  \begin{equation}
    \beta_\kappa \defeq \frac{\alpha \kappa^{3p}}{2 K_1}
    \,,
    \quad\text{and}\quad
      \norm{f}_{\beta_\kappa} \defeq \sup_x \frac{\abs{f(x)}}{1+\beta_\kappa V(x)}\,,
  \end{equation}
  where~$K_1$ and~$\alpha$ are the constants from Lemma~\ref{l:harrisConditions}.

  Following the proof of Lemma~3.3 in~\cite{CoopermanIyerEA25}, we note that the dependence on~$\kappa$ only enters in the inequality immediately preceding (8.1), through
  \begin{equation}\label{e:emKappa}
    \norm{e_{m'}^{(2)}}_{\beta_\kappa}\,,
    \quad\text{and}\quad
    \int (1+\beta_\kappa V)d\pi^{(2)}
    \,.
  \end{equation}
  Here $e_m^{(2)}:\T^{2, (2)} \to \C$ is the eigenfunction defined by
  \begin{equation}
    e_m^{(2)}(x, y) = e^{2\pi i m \cdot (x - y)}
    \,.
  \end{equation}
  We note that both terms in~\eqref{e:emKappa} satisfy the~$\kappa$-independent bound
  \begin{equation}
    \norm{e^{(2)}_m}_{\beta_\kappa}
      \leq \norm{e^{(2)}_m}_{L^\infty} = 1 \quad\text{and}\quad \int (1+\beta_\kappa V)d\pi^{(2)} \leq 2\,,
  \end{equation}
  for sufficiently small $\kappa$.
  Hence the proof of Lemma~3.3 in~\cite{CoopermanIyerEA25} will go through unchanged for all sufficiently small~$\kappa$ and yield~\eqref{e:ukappa-mixrate-kappa} and~\eqref{e:Dbound}.
\end{proof}

\section{Lower bounds on the dissipation time (Proposition \ref{p:tdis-lower})}\label{s:slowm}

We will now prove a lower bound on the dissipation time when the advecting velocity field is small in~$W^{-1, \infty}$, as stated in Proposition~\ref{p:tdis-lower}.
The proof is a direct argument based on energy methods. We note a somewhat similar argument is made using an SDE focused approach in~\cite{ColomboCrippaEA22} exploiting the ``It\^o-Tanaka trick''. 
\begin{proof}[Proof of Proposition~\ref{p:tdis-lower}]
  Since~\eqref{e:w-1inftybound} assumes a uniform in time~$L^\infty$ bound on~$H^\kappa$, it is enough to show~\eqref{e:tdisLower} holds for~$s = 0$.
Define $\theta^\kappa$ as the solution to the advection-diffusion equation~\eqref{e:ad} and $\varphi^\kappa$ as the solution to the heat equation 
\begin{equation}\label{e:heat}
    \partial_t \varphi^\kappa - \kappa\Delta \varphi^\kappa = 0\,.
\end{equation}
  Using~\eqref{e:w-1inftybound} and the fact that~$H$ is skew-symmetric, we note
\begin{equation}
 u^\kappa \cdot \nabla \theta^\kappa = \nabla \cdot (H^\kappa \cdot \nabla \theta^\kappa)\,. 
\end{equation}

Let $w = \theta^\kappa - \varphi^\kappa$.
Subtracting equations~\eqref{e:ad} and~\eqref{e:heat} gives
\begin{equation}
    \partial_t w + \nabla \cdot (H^\kappa \cdot \nabla w) + \nabla \cdot (H^\kappa \cdot \nabla \varphi^\kappa) = \kappa \Delta w\,.
\end{equation}
Multiplying by~$w$ and integrating gives
\begin{align}
    \frac{1}{2}\partial_t \norm{w}_{L^2}^2 &\leq -\kappa\norm{\nabla w}_{L^2}^2 + \norm{H^\kappa}_{L^\infty}\norm{\nabla w}_{L^2}\norm{\nabla \varphi^\kappa}_{L^2}\\
    &\leq -\frac{\kappa}{2}\norm{\nabla w}_{L^2}^2 + \frac{1}{2\kappa}\norm{H^\kappa}_{L^\infty}^2\norm{\nabla \varphi^\kappa}_{L^2}^2\\
    \label{e:wdiff}&\leq \frac{1}{2} C_0^2 \kappa\norm{\nabla \varphi^\kappa_0}_{L^2}^2\,.
\end{align}

Now we choose
\begin{equation}
  \varphi^\kappa_0 (x) = \theta^\kappa_0(x) = e_1\,,
\end{equation}
  where $e_1 = \sin( x_1 )$ is the first eigenfunction of $-\Delta$ on $\T^d$ corresponding to the eigenvalue~$\lambda_1 = 1$.  We note $w_0 =0$ and
\begin{equation}
  \norm{w_t}_{L^2}^2 \overset{\eqref{e:wdiff}}{\leq} C_0^2 \norm{\grad e_1}_{L^2}^2 \kappa t
  = C_0^2 \kappa t
  \,.
\end{equation}
  Thus, if~$C_1$ is defined by~\eqref{e:tdisLower} and~$t \leq C_1 / \kappa$, then
\begin{equation}
    \norm{\varphi^\kappa_t}_{L^2}^2 = e^{-2\kappa \iffalse\lambda_1\fi t} \geq \frac{3}{4}\,,
    \quad\text{and}\quad
    \norm{w_t}_{L^2}^2 \leq \frac{1}{8}\,.
\end{equation}
Using the triangle inequality,
\begin{equation}
    \norm{\theta^\kappa_t}_{L^2}^2 \geq \frac{1}{2}\norm{\varphi^\kappa_t}_{L^2}^2 - \norm{w_t}_{L^2}^2 \geq \frac{1}{4}\,,
\end{equation}
  which forces~$\tdis^0( u^\kappa, \kappa ) \geq t$.
  This proves~\eqref{e:tdisLower} concluding the proof.
\end{proof}

\section{The dissipation time of mixing flows}\label{s:tdiss-bound}

In this section we prove Proposition~\ref{prop:dissipation at mixing time} and Corollary~\ref{cor:exponential mixing dissipation enhancement}. We start with an estimate on the closeness of the solution to a drift-diffusion equation to the solution to the associated transport equation. We use throughout this section the solution operator notation given by Definition~\ref{def:solution operator}.

\begin{proposition}
Let $u \in L^\infty([0,\infty); W^{1,\infty}(\T^d))$ be a divergence-free velocity field and let $\theta \in L^2(\T^d)$. Then denote $\theta^0_t \defeq \sol_{0,t}^0 \theta$ and $\theta^\kappa_t \defeq \sol_{0,t}^\kappa \theta.$ We have the estimate.
    \begin{equation}
    \label{eq:closeness estimate} 
    \|\theta^\kappa_t -\theta^0_t\|_{L^2} \leq e^3\sqrt{\kappa} \Big( \|\nabla \theta\|_{L^2} + \sqrt{\|\nabla u\|_{L^\infty_{t,x}}+1}\int_0^t \|\nabla \theta^\kappa_s\|_{L^2}\,ds \Big).
        \end{equation}
\end{proposition}
\begin{proof}
    Let $\phi^\kappa_t \defeq \theta^\kappa_t - \theta^0_t$ so that
\[\partial_t \phi^\kappa - \kappa \Delta \phi^\kappa + u \cdot \nabla \phi^\kappa = \kappa \Delta \theta^0.\]
Then
\begin{align}
  \partial_t \|\phi^\kappa_t\|_{L^2}^2
    &\leq - 2\kappa \|\nabla \phi^\kappa_t\|_{L^2}^2 + 2\kappa \|\nabla \phi^\kappa_t\|_{L^2} \|\nabla \theta^0_t\|_{L^2}
  \\
    &\leq \kappa \|\nabla \theta^0_t\|_{L^2}^2
    \leq \kappa e^{2 t\|\nabla u\|_{L^\infty_{t,x}}} \|\nabla \theta\|_{L^2}^2
    \,.
\end{align}
Thus we have the bound
\[\|\sol_{0,t}^\kappa - \sol^0_{0,t}\|_{H^1 \to L^2} \leq \sqrt{\kappa t} e^{t \|\nabla u\|_{L^\infty_{t,x}}}.\]
Then for any $n \in \N,$
\begin{align*}
    \|\theta^\kappa_t -\theta^0_t\|_{L^2} &= \|(\sol^\kappa_{0,t} - \sol^0_{0,t}) \theta\|_{L^2}
    \\&=  \|\big((\sol^\kappa_{\frac{n-1}{n}t,t}  - \sol^0_{\frac{n-1}{n}t ,t})\sol^\kappa_{0,\frac{n-1}{n} t} + \sol^0_{\frac{n-1}{n}t ,t}(\sol^\kappa_{0,\frac{n-1}{n} t} - \sol^0_{0,\frac{n-1}{n}t})\big) \theta\|
    \\&\leq  \sqrt{\kappa} \sqrt{\frac{n}{t}}  \exp\Big(\frac{t\|\nabla u\|_{L^\infty_{t,x}}}{n}\Big) \frac{t}{n} \|\nabla \theta^\kappa_{\frac{n-1}{n}t}\|_{L^2} + \|\theta^\kappa_{\frac{n-1}{n}t} - \theta^0_{\frac{n-1}{n}t}\|_{L^2}
    \\&= \sqrt{\kappa} \sqrt{\frac{n}{t}} \exp\Big(\frac{t \|\nabla u\|_{L^\infty_{t,x}}}{n}\Big)  \frac{t}{n}\sum_{j=0}^{n-1} \|\nabla \theta^\kappa_{tj/n}\|_{L^2}. 
\end{align*}
Then we note that for $j \geq 1,0 \leq  s \leq t/n$, we have that
\[\|\nabla \theta^\kappa\|_{tj/n} \leq e^{s \|\nabla u\|_{L^\infty_{t,x}}} \|\nabla \theta^\kappa_{tj/n-s}\|_{L^2} \leq \exp\Big(\frac{t \|\nabla u\|_{L^\infty}}{n}\Big) \|\nabla \theta^\kappa_{tj/n-s}\|_{L^2},\]
and so
\[\frac{t}{n}\sum_{j=1}^{n-1} \|\nabla \theta^\kappa_{tj/n}\|_{L^2} \leq \exp\Big(\frac{t \|\nabla u\|_{L^\infty}}{n}\Big) \int_0^{t} \|\nabla \theta^\kappa_s\|_{L^2}\,ds.
\]
Combining the displays, we get that 
\[  \|\theta^\kappa_t -\theta^0_t\|_{L^2} \leq \sqrt{\kappa} \sqrt{\frac{n}{t}} \exp\Big(\frac{2t \|\nabla u\|_{L^\infty_{t,x}}}{n}\Big) \Big( \frac{t}{n} \|\nabla \theta\|_{L^2} + \int_0^t \|\nabla \theta^\kappa_s\|_{L^2}\,ds \Big).\]
We now choose
\[n \defeq \lceil t (\|\nabla u\|_{L^\infty_{t,x}} + 1)\rceil,\]
giving
\[  \|\theta^\kappa_t -\theta^0_t\|_{L^2} \leq e^2\sqrt{\kappa} \Big( \|\nabla \theta\|_{L^2} + \big(\sqrt{\|\nabla u\|_{L^\infty_{t,x}}+1} +t^{-1/2}\big)\int_0^t \|\nabla \theta^\kappa_s\|_{L^2}\,ds \Big).\]
Then if $t \leq (\|\nabla u\|_{L^\infty_{t,x}}+1)^{-1}$, then
\[t^{-1/2}\int_0^t \|\nabla \theta^\kappa_s\|_{L^2}\,ds \leq e t^{1/2} \|\nabla \theta\|_{L^2} \leq e \|\nabla \theta\|_{L^2},\]
otherwise $t \geq (\|\nabla u\|_{L^\infty_{t,x}}+1)^{-1}$ and so $t^{-1/2} \leq \sqrt{\|\nabla u\|_{L^\infty_{t,x}}+1}$. In either case, we can conclude.
\end{proof}

We now prove Proposition~\ref{prop:dissipation at mixing time}. We need to define the following orthogonal Fourier projectors. For $R \in [0,\infty)$, we define the orthogonal projection $\Pi_{\leq R}, \Pi_{>R} : \dot L^2(\T^d) \to \dot L^2(\T^d)$ so that if
\[f(x) = \sum_{k \in \Z^d} e^{i k \cdot x} \hat f(k) \quad \text{then} \quad \Pi_{\leq R} f(x) \defeq \sum_{k \in \Z^d, |k| \leq R} e^{ik \cdot x} \hat f(k),\]
and $\Pi_{>R} = 1 - \Pi_{\leq R}.$

The key idea of the proof is to case split according to whether the drift-diffusion solution $\theta^\kappa_t$ stays close in $L^2(\T^d)$ to some solution to the transport equation or if it separates from it. If it stays close, then the mixing of the transport equation ensures, at the relevant time scale, that most of the $L^2$ mass of the transport solution---and hence by the $L^2$ closeness, most of the mass of $\theta^\kappa_t$---is on high Fourier modes, thus inducing $L^2$-norm dissipation by the energy identity. On the other hand, if the transport solution separates from $\theta^\kappa_t$, then~\eqref{eq:closeness estimate} gives a lower bound on the $L^1_t H^1_x$ norm of $\theta^\kappa_t$. The energy identity gives that $L^2$ dissipation is however governed by the $L^2_t H^1_x$ norm of the $\theta^\kappa_t$, so we need to apply H\"older's inequality, giving an inverse time scale factor. It is this inverse time scale---coming from comparing the $L^2$ norm in time to the $L^1$ norm in time---that causes us to get the dissipation time of an exponentially mixing flow to be $\abs{\log \kappa}^2$ instead of $\abs{\log \kappa}$.

This proof technique is largely analogous to that of~\cite{CotiZelatiDelgadinoEA20} however uses a slightly more refined (or ``less interpolated'') estimate of the distance of a transport solution and an advection-diffusion solution in the form of~\eqref{eq:closeness estimate}. Additionally, taking a slightly different approach gives more information on shorter times, essentially showing some linear decay of the $L^2$ norm on the time interval $[0,\tdis].$

\begin{proof}[Proof of Proposition~\ref{prop:dissipation at mixing time}]
    Let $\theta_0 \in \dot L^2$ and without loss of generality suppose that $\|\theta_0\|_{L^2} =1$. Let $\theta^\kappa_t \defeq \sol^\kappa_{0,t} \theta_0$. To conclude, it suffices to show that
    \[\|\theta^\kappa_{\tau_\kappa}\|_{L^2}^2 \leq 1 - A_\kappa.\]
    We suppose for the sake of contradiction that $\|\theta^\kappa_{\tau_\kappa}\|_{L^2}^2 > 1 - A_\kappa \geq \frac{1}{2}$. Then by the energy identity, we have that
    \[1 - A_\kappa < \|\theta^\kappa_{\tau_\kappa}\|_{L^2}^2 = 1 - 2\kappa \int_0^{\tau_\kappa} \|\nabla \theta^\kappa_s\|_{L^2}^2\,ds.\]
    Thus
    \[2\kappa \int_0^{\tau_\kappa/3} \|\nabla \theta^\kappa_s\|_{L^2}^2\,ds \leq  2\kappa \int_0^{\tau_\kappa} \|\nabla \theta^\kappa_s\|_{L^2}^2\,ds < A_\kappa.\]
    As such, we can choose some $t_0 \in [0,\tau_\kappa/3]$ such that
    \begin{equation}
    \label{eq:small H1}
    \|\nabla \theta^\kappa_{t_0}\|_{L^2} \leq \sqrt{\frac{3A_\kappa}{2\kappa \tau_\kappa}}.\end{equation}
    Define $\theta^0_{t_0,t} \defeq \sol^0_{t_0,t} \theta^\kappa_{t_0}.$ We split into two cases.

    \restartcases
    \case[{For all $t \in [2\tau_\kappa/3,\tau_\kappa],\; \|\theta^\kappa_t - \theta^0_{t_0,t}\|_{L^2} \leq \tfrac{1}{4}$}]
    By the definition of mixing rate, for all $t \in [2\tau_\kappa/3, \tau_\kappa]$, we see
    \[\|\theta^0_{t_0,t}\|_{H^{-1}} \leq H(\tau_\kappa) \|\nabla \theta^\kappa_{t_0}\|_{L^2} \leq H(\tau_\kappa)\sqrt{\frac{3A_\kappa}{2\kappa \tau_\kappa}} \defeq B_\kappa^{-1}.\]
    Then
    \[\|\Pi_{> B_\kappa/4} \theta^0_{t_0,t}\|_{L^2}^2 = \|\theta^0_{t_0,t}\|_{L^2}^2 - \|\Pi_{\leq B_\kappa/4} \theta^0_{t_0,t}\|_{L^2}^2 \geq \|\theta^\kappa_{\tau_\kappa}\|_{L^2}^2 - \frac{B_\kappa^{2}}{16} \|\theta^0_{t_0,t}\|_{H^{-1}}^2 \geq \frac{1}{9}.\]
    Then, still for $t \in [2\tau_\kappa/3,\tau_\kappa],$
    \begin{align*}\|\nabla \theta^\kappa_t\|_{L^2} &\geq \frac{B_\kappa}{4} \|\Pi_{> B_\kappa /4} \theta^\kappa_t\|_{L^2} 
    \\&\geq \frac{B_\kappa }{4} \big(\|\Pi_{> B_\kappa /4} \theta^0_{t_0,t}\|_{L^2} - \|\Pi_{> B_\kappa /4} (\theta^\kappa_t - \theta^0_{t_0,t})\|_{L^2}\big) 
    \\&\geq \frac{B_\kappa }{4} \Big(\frac{1}{3} - \|\theta^\kappa - \theta^0_{t_0,t}\|_{L^2}\Big)\geq \frac{B_\kappa }{48}.
    \end{align*}
    Thus
    \[A_\kappa > 2 \kappa \int_0^{\tau_\kappa} \|\nabla \theta^\kappa_s\|_{L^2}^2\,ds \geq 2 \kappa \int_{2\tau_\kappa/3}^{\tau_\kappa} \|\nabla \theta^\kappa_s\|_{L^2}^2\,ds \geq \frac{2\kappa B_\kappa^{2} \tau_\kappa}{3 \cdot 48^2} = \Big(\frac{\kappa \tau_\kappa}{72 H(\tau_\kappa)}\Big)^2 A_\kappa^{-1},\]
    or using the definitions of $\tau_\kappa, A_\kappa,$
    \[ \frac{1}{72  \cdot 2^8 (\|\nabla u\|_{L^\infty_{t,x}}+1)\tau_\kappa} \leq \frac{\kappa \tau_\kappa}{72 H(\tau_\kappa)}<A_\kappa =  \frac{1}{2^{15}(\|\nabla u\|_{L^\infty_{t,x}} + 1) \tau_\kappa},\]
    which is a contradiction as $72 \cdot 2^8 < 2^{15}.$

\case[{For some $ t_1 \in [2\tau_\kappa/3,\tau_\kappa],$ such that $\|\theta^\kappa_{t_1} - \theta^0_{t_0,t_1}\|_{L^2} \geq \frac{1}{4}$}]
    By~\eqref{eq:closeness estimate} and~\eqref{eq:small H1}, we see
    \[
    \frac{1}{4} \leq \|\theta^\kappa_{t_1} - \theta^0_{t_0,t_1}\|_{L^2} \leq e^3\sqrt{\kappa} \Big( \sqrt{\frac{3A_\kappa}{2\kappa \tau_\kappa}} + \sqrt{\|\nabla u\|_{L^\infty_{t,x}}+1}\int_0^{\tau_\kappa} \|\nabla \theta^\kappa_s\|_{L^2}\,ds \Big).
    \]
    Then we note that by the definition of $A_\kappa$ and that $\tau_\kappa \geq 2,$
    \[e^3 \sqrt{\frac{3 A_\kappa}{2 \tau_\kappa}} \leq \frac{1}{4(\|\nabla u\|_{L^\infty_{t,x}} + 1)^{1/2} \tau_\kappa} \leq \frac{1}{8}. \]
    Combining the previous two equations, using the definition of $A_\kappa$, and applying H\"older's inequality, we have
    \begin{align*}\tau_\kappa A_\kappa = \frac{1}{2^{15} \big(\|\nabla u\|_{L^\infty_{t,x}}+1\big)} &\leq  2\kappa\Big(\int_0^{\tau_\kappa} \|\nabla \theta^\kappa_s\|_{L^2}\,ds \Big)^2 
   \\&\leq 2\tau_\kappa \kappa\int_0^{\tau_\kappa} \|\nabla \theta^\kappa_s\|_{L^2}^2\,ds <   \tau_\kappa A_\kappa,
    \end{align*}
    which is a contradiction. Thus we conclude in both the cases.
\end{proof}

We conclude by proving Corollary~\ref{cor:exponential mixing dissipation enhancement} as a consequence of Proposition~\ref{prop:dissipation at mixing time}.

\begin{proof}[Proof of Corollary~\ref{cor:exponential mixing dissipation enhancement}]
    Let $\theta_0 \in \dot L^2$ arbitrary, and let $\theta^\kappa_t$ denote the solution to~\eqref{e:ad} with $\theta^\kappa_0 = \theta_0.$ 
    Let
    \[ T_\kappa\defeq 2^{24}(\|\nabla u\|_{L^\infty_{t,x}} + 1)\bigg(1 +  2^{24} \Big(\frac{p}{\gamma}\Big)^4(\|\nabla u\|_{L^\infty_{t,x}} + 1) + \gamma^{-2}\big((\log D)^2 + \abs{\log \kappa}^2\big)  \bigg).\]
    Then for any $0 \leq r \leq T_\kappa$, we define as in Proposition~\ref{prop:dissipation at mixing time},
        \begin{align*}
    H^r(T) &\defeq \sup_{0 \leq s \leq T/3 \leq t \leq T} h(r+s,t)
    \\\tau^r_\kappa &\defeq \inf \big\{t \in [2,\infty) : t^{-2} H^r(t) \leq 2^{8}(\|\nabla u\|_{L^\infty_{t,x}}+1)\kappa\big\}, \\A^r_\kappa &\defeq \frac{1}{2^{15}(\|\nabla u\|_{L^\infty_{t,x}} + 1) \tau^r_\kappa}.
    \end{align*}
     so that
    \begin{equation}
    \label{eq:enhanced dissipation iterable}
    \|\theta_{r+\tau^r_\kappa}^\kappa\|_{L^2} \leq \sqrt{1 - A^r_\kappa} \|\theta^\kappa_r\|_{L^2}.\end{equation}
     Note that, since $r \leq T$, using the definition of $h(s,t)$, we can directly verify that
    \[\tau_\kappa^r \leq 2 + \frac{3}{\gamma} \big( \log D + p \log(4 T_\kappa)+ \abs{\log \kappa}\big) \defeq \sigma_{\kappa} \leq T_\kappa,\]
     and thus
    \[A^s_\kappa \geq \frac{1}{2^{15}(\|\nabla u\|_{L^\infty_{t,x}} + 1)  \sigma_{\kappa}}.\]
    Thus for any $n \in \N$ such that $n \sigma_{\kappa} \leq T$, we have by iterating~\eqref{eq:enhanced dissipation iterable} that
    \[\|\theta^\kappa_T\|_{L^2} \leq \Big(1 -  \frac{1}{2^{15}(\|\nabla u\|_{L^\infty_{t,x}} + 1)  \sigma_{\kappa}}\Big)^{n/2} \|\theta_0\|_{L^2}.\]
    Using that $(1- x)^{x^{-1}} \leq \frac{1}{2}$, we thus get that $\|\theta^\kappa_{T_\kappa}\|_{L^2} \leq \frac{1}{2} \|\theta_0\|_{L^2}$---allowing us to conclude---provided
    \[2^{17}(\|\nabla u\|_{L^\infty_{t,x}} + 1)  \sigma_{\kappa}^2 \leq T_\kappa.\]
    Using the definition of $\sigma_{\kappa}$, we see this is implied provided
    \begin{equation}
    \label{eq:enhanced dissipation suffices}
    1 + \gamma^{-2} \big( (\log D)^2 + \abs{\log \kappa}^2\big) \leq \frac{T_\kappa}{2^{23}(\|\nabla u\|_{L^\infty_{t,x}} + 1)} - \Big(\frac{p}{\gamma} \log(4 T_\kappa)\Big)^2.\end{equation}
    We first note that since 
    \[T_\kappa \geq 2^{48} p^4 \gamma^{-4}(\|\nabla u\|_{L^\infty_{t,x}} + 1)^2,\]
    we have that
    \[\frac{T_\kappa}{2^{23}(\|\nabla u\|_{L^\infty_{t,x}} + 1)} - \Big(\frac{p}{\gamma} \log(4 T_\kappa)\Big)^2 \geq \frac{T_\kappa}{2^{24}(\|\nabla u\|_{L^\infty_{t,x}} + 1)}.\]
    Then using that 
    \[T_\kappa \geq 2^{24}(\|\nabla u\|_{L^\infty_{t,x}} + 1)(1+ \gamma^{-2}((\log D)^2+ \abs{\log \kappa}^2)),\]
    we conclude~\eqref{eq:enhanced dissipation suffices} and thus the proof.
\end{proof}

\appendix

\bibliographystyle{halpha-abbrv}
\bibliography{gautam-refs1,gautam-refs2,preprints}
\end{document}

%% file: preamble.tex
\usepackage{marktext}
\usepackage{gimac}

\newcommand{\GI}[2][]{\sidenote[colback=yellow!20]{\textbf{GI\xspace #1:} #2}}

\newcommand{\tdis}{t_\mathrm{dis}}

\newcommand{\sol}{\mathcal{T}}

\DeclareMathOperator{\dist}{dist}

\DeclareMathOperator{\Leb}{Leb}
\DeclareMathOperator{\Unif}{Unif}
\DeclareMathOperator{\Law}{Law}

\newcommand{\vertiii}[1]{{\left\vert\kern-0.25ex\left\vert\kern-0.25ex\left\vert #1
    \right\vert\kern-0.25ex\right\vert\kern-0.25ex\right\vert}}

\renewcommand{\setminus}{-}